\documentclass{article}

\usepackage{PRIMEarxiv}

\usepackage[utf8]{inputenc} 
\usepackage[T1]{fontenc}    
\usepackage{hyperref}       
\usepackage{url}            
\usepackage{booktabs}       
\usepackage{amsfonts}       
\usepackage{nicefrac}       
\usepackage{microtype}      
\usepackage{lipsum}
\usepackage{fancyhdr}       
\usepackage{graphicx}       
\graphicspath{{media/}}     

\usepackage{natbib}

\usepackage{amsmath,amsfonts,bm}

\def\eqref#1{equation~\ref{#1}}

\def\1{\bm{1}}

\usepackage{amsmath,amssymb,amsfonts,mathtools,bm}
\usepackage{amsthm}
\usepackage{microtype}
\usepackage{booktabs,multirow,array,tabularx}
\usepackage{algorithm}
\usepackage{algpseudocode}
\usepackage{enumitem}
\usepackage{graphicx}
\usepackage{wrapfig}
\usepackage{rotating}
\usepackage{url}
\usepackage{hyperref}
\usepackage{xcolor}
\usepackage{makecell}
\usepackage{pifont}
\usepackage{etoolbox}
\hypersetup{colorlinks=true,citecolor=blue!60!black,linkcolor=blue!60!black,urlcolor=blue!60!black}
\setlist[itemize]{leftmargin=*,topsep=2pt,itemsep=1pt,parsep=0pt}
\setlist[enumerate]{leftmargin=*,topsep=2pt,itemsep=1pt,parsep=0pt}

\DeclareMathAlphabet{\mathsfit}{\encodingdefault}{\sfdefault}{m}{sl}
\SetMathAlphabet{\mathsfit}{bold}{\encodingdefault}{\sfdefault}{bx}{n}

\newcommand{\E}{\mathbb{E}}

\newcommand{\R}{\mathbb{R}}

\newtheorem{theorem}{Theorem}

\newtheorem{lemma}[theorem]{Lemma}

\newtheorem{assumption}[theorem]{Assumption}
\newtheorem{remark}[theorem]{Remark}
\theoremstyle{definition}

\newcommand{\Prb}{\mathbb{P}}

\newcommand{\ip}[2]{\left\langle #1,#2\right\rangle}

\newcommand{\clip}{\operatorname{clip}}

\definecolor{tablecheckgreen}{RGB}{0,150,70}
\definecolor{tablecrossred}{RGB}{210,35,45}
\newcommand{\yes}{{\color{tablecheckgreen}\ding{51}}}
\newcommand{\no}{{\color{tablecrossred}\ding{55}}}

\DeclareMathOperator{\diam}{diam}
\DeclareMathOperator{\Ber}{Ber}

\usepackage{float}

\floatstyle{ruled}
\newfloat{procedure}{tbp}{lop}
\floatname{procedure}{Procedure}
  
\title{Accelerated Stochastic Method under $(H_0,H_1)$-Smoothness and Heavy-Tailed Noise
}

\author{
Aleksandr Lobanov\\
MSU AI Center\\
\texttt{lobanovav@my.msu.ru}
 \And
 Darina~Dvinskikh\\
 HSE University \\
\texttt{dmdvinskikh@hse.ru}\\
   \And
  Alexander Gasnikov\\
Innopolis University\\
\texttt{gasnikov@yandex.ru} 
}

\begin{document}
\maketitle

\begin{abstract}
We develop an accelerated stochastic method for convex $(H_0,H_1)$-smooth optimization under heavy-tailed noise. The unbiased oracle has a finite $p$-th noise moment, $1<p\le2$, with constant, gradient-dependent, and gap-dependent terms. Our method combines accelerated updates with clipping, projection, and phase restarts. Each iteration uses a single stochastic-gradient sample, without minibatching. We prove high-probability convergence to any desired accuracy despite clipping bias. The deterministic terms retain square-root dependence on both $H_0$ and $H_1$. In our three-component noise model, strong convexity leaves a polynomial accuracy cost only for the constant component; without it, the dependence is logarithmic even for $p<2$. More generally, for noise moments scaling as $(f(x)-f^\star)^\alpha$, the stochastic cost becomes logarithmic at $\alpha=p$ for convex objectives and $\alpha=p/2$ under strong convexity.
\end{abstract}


\section{Introduction}
Training modern AI models is not only stochastic; it can also be intermittently unstable. Gradient clipping was introduced as an effective response to exploding gradients in recurrent networks \citep{pascanu2013rnn}, and it is now a standard safeguard in large-scale learning. Yet instability remains costly at scale: the largest PaLM model experienced roughly twenty irregular loss spikes despite gradient clipping, requiring checkpoint recovery and skipped data batches \citep{chowdhery2023palm}. Clipping is also not theoretically innocuous: applied to stochastic gradients, it generally creates bias whose effect depends on the threshold and on the noise model \citep{koloskova2023clipping}. These observations expose a tension between acceleration and robustness.

Two forms of nonuniformity make this tension particularly relevant to machine learning. First, the local curvature of a neural-network loss can vary substantially along training and correlate with the gradient scale, motivating $(L_0,L_1)$ and more general smoothness models \citep{zhang2020clipping,li2023generalized}. Recent work on $(H_0,H_1)$-smoothness treats curvature as loss-dependent and shows that the admissible local step size can increase as the objective gap decreases. This model provides theoretical and empirical explanations of learning-rate warmup on language and vision models \citep{liu2025warmup,alimisis2025warmup}. Second, studies across architectures, datasets, and discriminative and generative models find strongly non-Gaussian, heavy-tailed stochastic-gradient noise, making large errors especially harmful to momentum \citep{simsekli2019tail,battash2024noise}. A global smoothness constant together with a state-independent noise bound can therefore be simultaneously pessimistic about both the geometry and the statistics encountered along the trajectory.

Beyond being heavy-tailed, stochastic-gradient noise can depend on the current iterate. In interpolation regimes, component gradients vanish at a common solution, while expected-smoothness conditions make their variability decrease with the objective gap \citep{ma2018interpolation,gower2019sgd}. This motivates separating residual, gradient-dependent, and gap-dependent noise in our oracle model.

Existing results address important parts of this picture separately. Accelerated clipping methods handle uniform heavy-tailed noise under classical smoothness \citep{gorbunov2020heavy,nguyen2023clipped}; generalized-smoothness analyses allow bounded-variance or affine-variance noise \citep{li2023generalized,yuhonglin2025rsag}; recent high-probability acceleration under generalized smoothness treats additive norm-sub-Gaussian noise \citep{dvinskikh2026localize}; and deterministic acceleration is known for the convex $(H_0,H_1)$ class \citep{lobanov2026h01}. The closest interpolation results under heavy tails analyze clipped or normalized SGD without accelerated high-probability guarantees for the mixed model \citep{lobanov2026bias}. This leads to the central question:
\begin{center}
\emph{Can clipping retain an accelerated optimization term while reaching arbitrary accuracy\\
under $(H_0,H_1)$-smoothness and mixed heavy-tailed noise?}
\end{center}

Our answer is affirmative for differentiable convex objectives. We first describe the geometry. With $F(x)=f(x)-f^\star$, the Hessian condition $\|\nabla^2 f(x)\|\le H_0+H_1F(x)$ allows high curvature at large gaps and lower curvature near the optimum \citep{vaswani2025armijo,liu2025warmup,alimisis2025warmup,lobanov2026bias,vaswani2026nonuniform}. We use the weaker first-order bound $f(y)\le f(x)+\langle\nabla f(x),y-x\rangle+(H_0+H_1F(x))\|y-x\|^2$ for $\|y-x\|\le r_H$, requiring only differentiability. For twice continuously differentiable convex functions, the conditions agree up to absolute constants \citep{lobanov2026h01}. Since this bound applies only within $r_H$, momentum can leave its valid range; preventing this is the first difficulty.

We next describe the stochastic model. Expected smoothness bounds $\mathbb E_\xi\|g(x,\xi)-g(x^\star,\xi)\|^2$ by a multiple of $F(x)$ \citep{gower2019sgd}; related models add constant, gap-dependent, or gradient-dependent terms \citep{khaled2023better,faw2022adaptivity,hong2024adagrad}, while finite-$p$ variants allow heavy tails \citep{zhanglin2026scale}. We combine them: for a fresh sample $\xi$ and $1<p\le2$, $g(x,\xi)$ satisfies $\mathbb E_\xi[g(x,\xi)]=\nabla f(x)$ and $\mathbb E_\xi\|g(x,\xi)-\nabla f(x)\|^p\le\sigma_0^p+\sigma_1^p\|\nabla f(x)\|^p+\sigma_2^pF(x)^{p/2}$. The terms describe noise that remains at the optimum, scales with the gradient, or decreases with the gap; $p<2$ permits infinite variance. The noise scale therefore changes between momentum queries. Clipping limits large samples but creates bias; controlling both is the second difficulty.

We address both difficulties with Restarted Clipped Stochastic Accelerated Gradient (RCSAG), illustrated in Figure~\ref{fig:rcsag-schematic}; Section~\ref{sec:phase} gives the full algorithm. Each phase starts with the gap bound $F(w)\le\Delta$ (panel (a)). Tracking the gap and distance to $x^\star$, we show that $u_t\in Q$, $F(u_t)\le4\Delta$, and all local-model displacements are at most $r_H$; this fixes phasewise bounds on curvature and noise. Panel (b) forms $u_t$ from $y_{t-1}$ and $z_{t-1}$ and clips its stochastic gradient; panel (c) projects $\widetilde z_t$ onto $Q$; and panel (d) averages $y_{t-1}$ with $z_t$ to form $y_t$. These bounds control random gradient errors and clipping bias without losing acceleration. With high probability, each phase halves the gap, and restarting from $y_T$ reaches the target accuracy.

\begingroup
\setlength{\intextsep}{3pt}
\setlength{\abovecaptionskip}{2pt}
\setlength{\belowcaptionskip}{0pt}
\begin{figure}[H]
\centering
\includegraphics[width=0.92\textwidth]{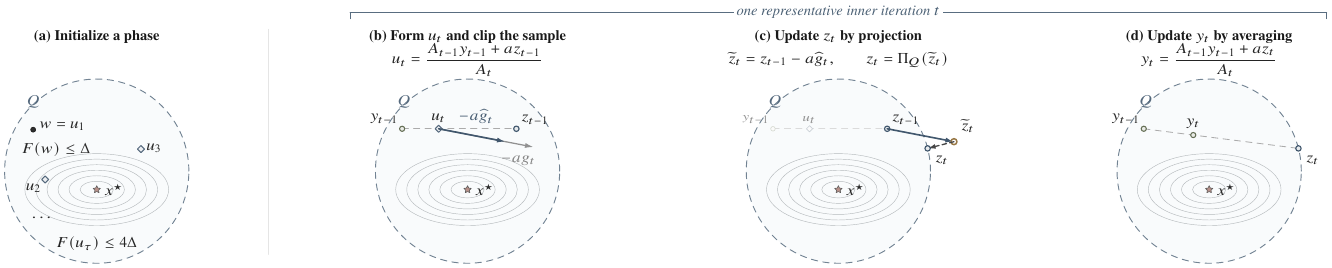}
\caption{RCSAG: initialize, form $u_t$, clip, project, average, and restart after $T$ steps.}
\label{fig:rcsag-schematic}
\end{figure}
\endgroup

The geometric part of our bounds retains the square-root dependence of deterministic acceleration on both curvature scales. Despite clipping bias, the method reaches any accuracy with high probability and separates the three noise components. Under strong convexity, only $\sigma_0$ forces polynomial dependence on $1/\varepsilon$; if $\sigma_0=0$, it is logarithmic even for $p<2$. For noise with $p$th moment bounded by $F(x)^\alpha$, the cumulative cost is polynomial, logarithmic, or summable, with thresholds $\alpha=p$ for convex and $\alpha=p/2$ for strongly convex objectives. On classical smooth subclasses, the dependence on $H_0$, $\sigma_0$, and $\sigma_2$ matches lower bounds up to logarithms, while the $\sigma_1$ rate remains conservative.

\paragraph{Contributions.}
Our main contributions are:
\begin{itemize}
\item \textbf{Accelerated geometry under mixed heavy-tailed noise} (Section~\ref{sec:main-results}). We establish high-probability guarantees for convex and strongly convex objectives under $(H_0,H_1)$-smoothness and mixed heavy-tailed noise. The noise-independent geometric terms retain the square-root dependence on both $H_0$ and $H_1$ of deterministic accelerated methods.
\item \textbf{Single-sample clipping without an error floor} (Sections~\ref{sec:phase} and~\ref{sec:main-results}). Although clipping biases each gradient estimate, RCSAG reaches any prescribed accuracy with high probability using one fresh stochastic gradient per inner iteration and no minibatching. The bounds separate noise that remains at the optimum, scales with the gradient, or decreases with the objective gap. Under strong convexity, when $\sigma_0=0$, the target-accuracy dependence remains logarithmic even for $p<2$.
\item \textbf{A noise-decay phase transition} (Section~\ref{sec:phase-transition}). For noise satisfying $\mathbb E_\xi\|g(x,\xi)-\nabla f(x)\|^p\lesssim F(x)^\alpha$, we characterize when its cumulative cost is polynomial, logarithmic, or summable. The critical exponent is $\alpha=p$ for convex objectives and $\alpha=p/2$ under strong convexity.
\end{itemize}

The paper focuses on convex objectives in order to expose the interaction between acceleration, nonuniform curvature, and state-dependent heavy tails with sharp objective-gap guarantees. Extending this mechanism to general nonconvex neural-network training is an important open direction.

\section{Related Work}\label{sec:related}
\vspace{-0.5em}
\paragraph{Optimization beyond a global smoothness constant.}
Classical acceleration uses one global smoothness constant, whereas generalized-smoothness models allow curvature to change during optimization. Under gradient-dependent $(L_0,L_1)$-smoothness, existing analyses explain clipping and normalization \citep{zhang2020clipping}, identify distinct convergence regimes for gradient descent and its variants \citep{lobanov2024linear}, and extend these results to stochastic and biased gradients, as well as zeroth-order methods \citep{lobanov2025power}. Acceleration has also been studied under broader generalized smoothness with bounded variance \citep{li2023generalized}, with variance depending on the current gap or gradient norm \citep{yuhonglin2025rsag}, and under $(L_0,L_1)$-smoothness with additive norm-sub-Gaussian noise \citep{dvinskikh2026localize}. For the $(H_0,H_1)$ model, \citet{lobanov2026bias} give non-accelerated guarantees in deterministic and interpolation settings, while \citet{lobanov2026h01} establish deterministic acceleration. Neither result gives accelerated guarantees for the mixed heavy-tailed model considered here. \emph{Our bounds retain the accelerated square-root dependence on both $H_0$ and $H_1$ while allowing mixed finite-$p$ heavy-tailed noise.}

\vspace{-0.5em}
\paragraph{Noise that decreases during optimization.}
A fixed global noise bound misses the fact that stochastic gradients may become more accurate near a solution. Some analyses assume that stochastic gradients approach their values at the optimum as the objective gap shrinks \citep{gower2019sgd,khaled2023better}, while related models combine a constant noise level with terms depending on the gradient or function value \citep{faw2022adaptivity,hong2024adagrad}. Finite-$p$ extensions allow such dependence under heavy tails \citep{zhanglin2026scale}. Under interpolation, \citet{lobanov2026bias} show that clipped and normalized SGD can avoid a clipping-induced error floor under generalized smoothness, while \citet{gupta2024agnes} retain acceleration for the special case of multiplicative finite-variance noise. Thus, decreasing noise can improve convergence, but previous guarantees treat only parts of the mixed model considered here. \emph{Our analysis handles $\mathbb E_\xi\|g(x,\xi)-\nabla f(x)\|^p\le\sigma_0^p+\sigma_1^p\|\nabla f(x)\|^p+\sigma_2^pF(x)^{p/2}$ without an error floor. When the noise decreases sufficiently fast near the optimum, it adds at most a logarithmic dependence on $1/\varepsilon$.}

\vspace{-1em}
\begin{table}[H]
\centering
\caption{\scriptsize Convex first-order guarantees and oracle complexity.}
\label{tab:related}
\small
\setlength{\tabcolsep}{3pt}
\renewcommand{\arraystretch}{1.08}
\resizebox{\textwidth}{!}{%
\begin{tabular}{llllccc}
\toprule
Work & Smoothness & Tail assumption & Oracle complexity & Varying noise & Accel. & HP\\
\midrule
\citet{gorbunov2020heavy} & $L$ & finite 2nd moment & $\widetilde O_\varepsilon(\varepsilon^{-1/2}+\varepsilon^{-2})$ & \no & \yes & \yes\\
\citet{nguyen2023clipped} & $L$ & finite $p$th moment & $\widetilde O_\varepsilon(\varepsilon^{-1/2}+\varepsilon^{-q})$ & \no & \yes & \yes\\
\citet{yuhonglin2025rsag} & $(L_0,L_1)$ & bounded / sub-Gaussian & $\widetilde O_\varepsilon(\varepsilon^{-1}+\varepsilon^{-2})$ & \yes & \no & \yes\\
\citet{dvinskikh2026localize} & $(L_0,L_1)$ & sub-Gaussian & $\widetilde O_\varepsilon(\varepsilon^{-1/2}+\varepsilon^{-2})$ & \no & \yes & \yes\\
\citet{lobanov2026bias} (det.) & $(H_0,H_1)$ & -- & $\widetilde O_\varepsilon(\varepsilon^{-1})$ & -- & \no & --\\
\citet{lobanov2026h01} (det.) & $(H_0,H_1)$ & -- & $\widetilde O_\varepsilon(\varepsilon^{-1/2})$ & -- & \yes & --\\
This work & $(H_0,H_1)$ & finite $p$th moment & $\widetilde O_\varepsilon(\varepsilon^{-1/2}+\varepsilon^{-q})^\dagger$ & \yes & \yes & \yes\\
\bottomrule
\end{tabular}
}
\vspace{-5pt}
\end{table}

Table~\ref{tab:related} compares the closest first-order results for convex objectives. The columns record smoothness, tail assumptions, oracle complexity, whether the noise varies along the trajectory, acceleration, and high-probability guarantees (HP). Here $q=p/(p-1)$, and $\widetilde O_\varepsilon$ shows only the dependence on $\varepsilon$, suppressing logarithms and $\varepsilon$-independent factors. A finite second moment still permits heavy-tailed noise, whereas sub-Gaussian noise is light-tailed. The symbols \yes, \no, and -- mean covered, not covered, and not applicable. For \citet{yuhonglin2025rsag}, the displayed general-noise guarantee has an $\varepsilon^{-1}$ optimization term and is therefore not marked as accelerated. In our worst-case mixed-noise entry ($\dagger$), the $\varepsilon^{-q}$ term comes only from $\sigma_0$; the $\sigma_1$ and $\sigma_2$ terms scale as $\varepsilon^{-q/2}$ in Theorem~\ref{thm:convex}. \emph{Among the listed results, ours is the only one that combines $(H_0,H_1)$-smoothness, finite-$p$ tails, noise depending on the current point, acceleration, and high-probability guarantees.}

\section{Problem Setup}\label{sec:setup}
We consider the optimization problem
\begin{equation}\label{eq:problem}
 f^\star:=\min_{x\in\R^d} f(x),
\end{equation}
where $f:\R^d\to\R$ maps a parameter vector to a scalar objective value. Throughout the paper, $\|\cdot\|$ is the Euclidean norm. We assume that the solution set $X^\star:=\arg\min_{x\in\R^d}f(x)$ is nonempty, fix an arbitrary $x^\star\in X^\star$ for the analysis, and write $F(x):=f(x)-f^\star$. Given an initial point $x_0$, an upper bound $F(x_0)\le\Delta_0$, a target accuracy $\varepsilon\in(0,\Delta_0]$, and a confidence level $\rho\in(0,1)$, our goal is to return $\widehat x$ such that $F(\widehat x)\le\varepsilon$ with probability at least $1-\rho$.

\subsection{Assumptions on the objective function}
We first specify the objective class considered throughout the paper.

\begin{assumption}[Convexity]\label{ass:convexity}
The function $f$ is differentiable and, for some $\mu\ge0$, satisfies
\[
 f(y)\ge f(x)+\ip{\nabla f(x)}{y-x}+\frac{\mu}{2}\|y-x\|^2,
 \qquad \forall x,y\in\R^d.
\]
\end{assumption}

Assumption~\ref{ass:convexity} is standard in the theory of accelerated first-order methods \citep{nesterov2004}. The case $\mu=0$ gives ordinary convexity, while $\mu>0$ gives strong convexity and is used only in Section~\ref{sec:strongly-convex}. The latter covers, for example, convex empirical-risk objectives with an $\ell_2$ regularizer.

Our main geometric assumption is the first-order form of $(H_0,H_1)$-smoothness used by the algorithm.

\begin{assumption}[$(H_0,H_1)$-smoothness]\label{ass:h01}
There exist $H_0,H_1\ge0$. Set $r_H:=((2\sqrt{3}+\sqrt{6})\sqrt{H_1})^{-1}$ when $H_1>0$ and $r_H:=\infty$ when $H_1=0$. Then, for all $x,y\in\R^d$ satisfying $\|y-x\|\le r_H$,
\[
 f(y)\le f(x)+\ip{\nabla f(x)}{y-x}
 +\bigl(H_0+H_1F(x)\bigr)\|y-x\|^2.
\]
\end{assumption}

Assumption~\ref{ass:h01} requires only differentiability. When $H_1=0$, it reduces to standard $L$-smoothness with $L=2H_0$. When $H_1>0$, the curvature scale $H_0+H_1F(x)$ may be large far from the optimum and approaches $H_0$ as the gap decreases. The upper bound is required only for $\|y-x\|\le r_H$. $(H_0,H_1)$-smoothness has been used to analyze line-search methods, learning-rate warmup, and first-order optimization beyond global smoothness \citep{vaswani2025armijo,liu2025warmup,alimisis2025warmup,lobanov2026bias,vaswani2026nonuniform,lobanov2026h01}.

\subsection{Assumptions on the stochastic gradient oracle}
We next introduce the notation used for the stochastic oracle. At a query point $x$, the algorithm observes a stochastic gradient $g(x,\xi_t)$ computed from a fresh sample $\xi_t$; the deterministic gradient is denoted by $\nabla f(x)$. Let $\mathcal F_t$ be the filtration generated by the samples through iteration $t$, and let $\E_t[\cdot]:=\E[\cdot\mid\mathcal F_{t-1}]$. Each query point at iteration $t$ is $\mathcal F_{t-1}$-measurable.

We impose the following condition on the conditional mean and $p$th moment of the stochastic gradient.

\begin{assumption}[Stochastic gradient oracle]\label{ass:oracle}
For some $1<p\le2$ and $\sigma_0,\sigma_1,\sigma_2\ge0$, every $\mathcal F_{t-1}$-measurable query $x$ admits a fresh sample $g(x,\xi_t)$ satisfying
\begin{equation}\label{eq:oracle}
 \E_t[g(x,\xi_t)]=\nabla f(x),\qquad
 \E_t\|g(x,\xi_t)-\nabla f(x)\|^p
 \le \sigma_0^p+\sigma_1^p\|\nabla f(x)\|^p+\sigma_2^pF(x)^{p/2}.
\end{equation}
\end{assumption}

The three terms allow noise that remains at the optimum, scales with the gradient, or decreases with the objective gap. When $\sigma_1=\sigma_2=0$, Assumption~\ref{ass:oracle} reduces to the finite-$p$ moment model standard in heavy-tailed stochastic optimization \citep{gorbunov2020heavy,nguyen2023clipped}; $p=2$ gives a second-moment bound, while $p<2$ does not require finite variance. Bounds depending on the gradient or objective gap are used to describe stochastic gradients that become more accurate during optimization \citep{gower2019sgd,faw2022adaptivity,khaled2023better,hong2024adagrad,zhanglin2026scale}. When $\sigma_0=0$, the bound vanishes at $x^\star$, capturing the decreasing noise found in interpolation settings \citep{ma2018interpolation,lobanov2026bias}. We write $q:=p/(p-1)$. One oracle call means one evaluation of $g(x,\xi_t)$ using one fresh sample.

\section{Restarted Clipped Stochastic Accelerated Gradient}\label{sec:phase}
Classical accelerated stochastic methods are typically tuned using a global smoothness constant and a uniform noise bound. Neither scale is fixed under Assumptions~\ref{ass:h01} and~\ref{ass:oracle}: curvature and noise may change along the trajectory, while the upper model in Assumption~\ref{ass:h01} is guaranteed only for displacements of length at most $r_H$. Consequently, momentum, clipping bias, and the validity of the upper model must be controlled simultaneously. We address this difficulty through gap-halving phases. Starting from a point $w$ satisfying $F(w)\le\Delta$, one phase is designed to return $w^+$ such that $F(w^+)\le\Delta/2$ with high probability. The bound $\Delta$ provides fixed upper bounds on curvature and noise during that phase; after contraction, the next phase uses the smaller bound $\Delta/2$. We implement this idea through Restarted Clipped Stochastic Accelerated Gradient (RCSAG). Algorithm~\ref{alg:rcsag} gives the outer restart scheme, and Procedure~\ref{proc:phase} specifies its accelerated clipped phase method.

\begingroup
\setlength{\belowcaptionskip}{1pt}
\begin{algorithm}[H]
\caption{Restarted Clipped Stochastic Accelerated Gradient (RCSAG)}\label{alg:rcsag}
\small
\begin{algorithmic}[1]
\Require $x_0,\Delta_0,\varepsilon,\rho$; phase-set rule $\mathcal Q$; model upper bounds $p,\sigma_i,H_i$.
\State Set $w_0\gets x_0$, $S\gets\lceil\log_2(\Delta_0/\varepsilon)\rceil$.
\For{$s=0,\ldots,S-1$}
  \State Set $\Delta_s\gets2^{-s}\Delta_0$, $\rho_s\gets6\rho/[\pi^2(s+1)^2]$.
  \State Choose $Q_s\gets\mathcal Q(w_s,\Delta_s)$ and set $D_s\gets\diam(Q_s)$.
  \State Update $w_{s+1}\gets\Call{\hyperref[proc:phase]{RCSAG-Phase}}{w_s,\Delta_s,Q_s,D_s,\rho_s}$.
\EndFor
\State \Return $w_S$.
\end{algorithmic}
\end{algorithm}
\endgroup

RCSAG (Algorithm~\ref{alg:rcsag}) does not evaluate the objective gap during the run. It only requires an initial upper bound $F(x_0)\le\Delta_0$; all later bounds are set deterministically as $\Delta_s=2^{-s}\Delta_0$. Appendix~\ref{app:init} shows how to construct a valid $\Delta_0$ with high probability from stochastic gradients alone, without knowing $f^\star$, evaluating function values, or computing an exact gradient.

Algorithm~\ref{alg:rcsag} starts from $w_0=x_0$ and runs $S=\lceil\log_2(\Delta_0/\varepsilon)\rceil$ phases. At phase $s$, $w_s$ is the current point, $\Delta_s=2^{-s}\Delta_0$ is an upper bound on $F(w_s)$, and $\rho_s$ is the allowed failure probability. The rule $\mathcal Q$ chooses an auxiliary convex set $Q_s$ containing $w_s$ and $x^\star$, and $D_s$ is its diameter. In the convex case, the same fixed set $Q$ is used in every phase. Under strong convexity, $Q_s$ is the shrinking ball $B(w_s,\sqrt{2\Delta_s/\mu})$. These sets control the iterates but do not constrain problem~\eqref{eq:problem}. RCSAG-Phase returns the next point $w_{s+1}$ and, with probability at least $1-\rho_s$, halves the current gap bound: $F(w_{s+1})\le\Delta_s/2=\Delta_{s+1}$. Thus each successful phase preserves $F(w_s)\le\Delta_s$. Since $\sum_{s\ge0}\rho_s\le\rho$, the final output satisfies $F(w_S)\le\varepsilon$ with probability at least $1-\rho$.

Algorithm~\ref{alg:rcsag} is related to the multistage AC-SA of \citet{ghadimi2013shrinking}, which uses shrinking domains, and to R-clipped-SSTM of \citet{gorbunov2020heavy}, which changes batch sizes and clipping thresholds across stages to exploit strong convexity. \citet{lobanov2026h01} use the same gap-halving schedule for deterministic $(H_0,H_1)$ acceleration, while \citet{dvinskikh2026localize} restart batched accelerated solves on shrinking balls under norm-sub-Gaussian noise. The outer wrapper is therefore standard. The remaining task is to realize its gap-halving step under $(H_0,H_1)$-smoothness and mixed heavy-tailed noise (see Assumptions~\ref{ass:h01} and~\ref{ass:oracle}). We now describe the phase method that does so.

\subsection{One phase of RCSAG}
We now describe one call to RCSAG-Phase (Procedure~\ref{proc:phase}). It starts from a point $w$ satisfying $F(w)\le\Delta$ and uses an auxiliary convex set $Q$ containing $w$ and $x^\star$, with $D=\diam(Q)<\infty$. Its goal is to return $y_T$ such that $F(y_T)\le\Delta/2$ with probability at least $1-\rho$.

The gap bound $\Delta$ determines phasewise bounds on curvature, gradient norm, and noise. Let $C_\star=19+12\sqrt{2}$ if $H_1>0$ and $C_\star=2$ otherwise; define $\mathcal H_\Delta:=H_0+H_1\Delta$, $G_\Delta:=\sqrt{32C_\star\Delta\mathcal H_\Delta}$, and $\varsigma_{p,\Delta}^p:=\sigma_0^p+\sigma_1^pG_\Delta^p+\sigma_2^p(4\Delta)^{p/2}$. Here $\mathcal H_\Delta$ is the curvature scale for the phase, while $G_\Delta$ and $\varsigma_{p,\Delta}$ bound the gradient and noise scales along the controlled trajectory, as justified below. For $\Lambda>0$, let $\clip_\Lambda(v):=v\min\{1,\Lambda/\|v\|\}$. Using $D$ and $\rho$ supplied by Algorithm~\ref{alg:rcsag}, together with the bounds $G_\Delta$ and $\varsigma_{p,\Delta}$, set $\Lambda:=\max\{2G_\Delta,(64\varsigma_{p,\Delta}^pD/\Delta)^{1/(p-1)}\}$, $R:=D\Lambda/\Delta$, $T:=\lceil1024(1+R)\log(8/\rho)\rceil$, and $a:=5D^2/(\Delta T)$. Only $T$, and hence $a$, depends on $\rho$. Theorem~\ref{thm:phase} justifies these choices, and Procedure~\ref{proc:phase} below gives the resulting updates using one fresh stochastic gradient per iteration.


\begingroup
\setlength{\belowcaptionskip}{1pt}

\begin{procedure}[h]
\caption{\textproc{RCSAG-Phase}}\label{proc:phase}
\small
\begin{algorithmic}[1]
\Require $w,\Delta,Q,D,\rho$; phase parameters $\Lambda,T,a$.
\State Set $y_0=z_0=w$, $A_0\gets D^2/\Delta$, and
       $A_t\gets A_0+ta$ for $t=1,\ldots,T$.
\For{$t=1,\ldots,T$}
    \State Set $u_t\gets(A_{t-1}y_{t-1}+az_{t-1})/A_t$.
    \State Query $g_t\gets g(u_t,\xi_t)$ using one sample.
    \State Clip $\widehat g_t\gets\clip_\Lambda(g_t)$.
    \State Update $z_t\gets\Pi_Q(z_{t-1}-a\widehat g_t)$.
    \State Update $y_t\gets(A_{t-1}y_{t-1}+az_t)/A_t$.
\EndFor
\State \Return $y_T$.
\end{algorithmic}
\end{procedure}

\endgroup

Procedure~\ref{proc:phase} maintains a candidate output $y_t$ and an auxiliary point $z_t\in Q$. At iteration $t$, it combines $y_{t-1}$ and $z_{t-1}$ to form the query point $u_t$, evaluates $g_t=g(u_t,\xi_t)$, and clips it at radius $\Lambda$ to obtain $\widehat g_t$, so that $\|\widehat g_t\|\le\Lambda$. It then takes the auxiliary step $z_{t-1}-a\widehat g_t$, projects it onto $Q$ to obtain $z_t$, and combines $z_t$ with $y_{t-1}$ to update $y_t$. The mixing fraction is $a/A_t$, where $A_t=A_0+ta$; the positive $A_0$ keeps this fraction small during the early iterations and thereby controls the query displacements.

Procedure~\ref{proc:phase} combines two established ideas. The coupling of $y_t$, $z_t$, and $u_t$ follows the classical accelerated construction of \citet{nesterov2004}. Clipped-SSTM uses a related coupling with clipped minibatch gradients under $L$-smoothness and a uniform second-moment bound \citep{gorbunov2020heavy}, while \citet{nguyen2023clipped} analyze clipped accelerated mirror descent under a uniform finite-$p$ noise bound. RCSAG-Phase instead adapts $\Lambda$ and $T$ to the current gap bound $\Delta$, uses one fresh sample per iteration, and combines the positive $A_0$ with projection onto the auxiliary set $Q$ to keep momentum queries within distance $r_H$. The individual ingredients are standard; our contribution is to combine them so that changing curvature, mixed noise, and clipping bias are controlled simultaneously without minibatching. The next subsection establishes safe queries and the bounds needed for the one-phase guarantee.

\subsection{Auxiliary Results for One Phase}
The phase parameters above use $G_\Delta$ and $\varsigma_{p,\Delta}$ as bounds at every query, so we must verify them along Procedure~\ref{proc:phase}. We track the potential $\Phi_t:=A_tF(y_t)+\frac12\|z_t-x^\star\|^2$, which combines the output gap with the auxiliary distance to a solution. Since $F(w)\le\Delta$, $w,x^\star\in Q$, and $A_0=D^2/\Delta$, we have $\Phi_0\le3D^2/2$; also $A_T=A_0+Ta=6D^2/\Delta$. What remains to prove is that $\Phi_t\le3D^2$ throughout the phase with high probability. On that event, $A_TF(y_T)\le\Phi_T\le3D^2$, hence $F(y_T)\le\Delta/2$. The lemmas below provide the needed bounds, and the next subsection completes the argument.

Set $\mathcal H(x):=H_0+H_1F(x)$. The first lemma bounds $\|\nabla f(x)\|$ in terms of $F(x)$ and $r_H$.

\begin{lemma}[Gradient bounds from the objective gap]\label{lem:selfbound}
Suppose Assumptions~\ref{ass:convexity} and~\ref{ass:h01} hold. Let $C_\star$ be the constant defined before Procedure~\ref{proc:phase}, and set $r_H^{-1}=0$ when $H_1=0$. Then, for every $x\in\R^d$,
\begin{equation}\label{eq:selfbounds}
 r_H^{-1}\|\nabla f(x)\|\le C_\star\mathcal H(x),
 \qquad
 \|\nabla f(x)\|^2\le 2C_\star F(x)\mathcal H(x).
\end{equation}
\end{lemma}

The next lemma turns these gradient bounds into bounds at the query $u_t$, conditional on $\Phi_{t-1}\le3D^2$. Appendix~\ref{app:phaseproof} shows that this condition holds throughout the phase with high probability.

\begin{lemma}[Safe accelerated queries]\label{lem:safe-queries}
Under Assumptions~\ref{ass:convexity} and~\ref{ass:h01}, run Procedure~\ref{proc:phase} with $F(w)\le\Delta$, $w,x^\star\in Q$, and the phase parameters $\Lambda,T,a$ specified above. If $\Phi_{t-1}\le3D^2$ for some $1\le t\le T$:
\begin{equation}\label{eq:safe-main}
 F(y_{t-1})\le3\Delta,\qquad F(u_t)\le4\Delta,\qquad \|\nabla f(u_t)\|\le G_\Delta.
\end{equation}
Moreover, $\|u_t-y_{t-1}\|\le r_H$ and $\|y_t-u_t\|\le r_H$.
\end{lemma}

Lemmas~\ref{lem:selfbound} and~\ref{lem:safe-queries} justify $G_\Delta$ and validate both local-model steps; the key bounds are \eqref{eq:selfbounds} and \eqref{eq:safe-main}. Assumption~\ref{ass:oracle} then gives $\E_t\|g(u_t,\xi_t)-\nabla f(u_t)\|^p\le\varsigma_{p,\Delta}^p$. Clipping bounds individual samples but may introduce bias, so we write $\widehat g_t-\nabla f(u_t)=\zeta_t+\beta_t$, where $\zeta_t:=\widehat g_t-\E_t[\widehat g_t]$ is centered and $\beta_t:=\E_t[\widehat g_t]-\nabla f(u_t)$ is the clipping bias.

\begin{lemma}[Bias and noise after clipping]\label{lem:clipping}
Suppose Assumption~\ref{ass:oracle} and the conditions of Lemma~\ref{lem:safe-queries} hold at some $t\in\{1,\ldots,T\}$. Then, for $\Lambda\ge2G_\Delta$,
\begin{equation}\label{eq:clip-summary}
 \|\zeta_t\|\le2\Lambda\quad\text{a.s.},\qquad
 \|\beta_t\|\le2\varsigma_{p,\Delta}^p\Lambda^{1-p},\qquad
 \E_t\|\zeta_t\|^2\le2\varsigma_{p,\Delta}^p\Lambda^{2-p}.
\end{equation}
\end{lemma}

The condition $\Lambda\ge2G_\Delta$ in Lemma~\ref{lem:clipping} motivates the $2G_\Delta$ term in our choice of clipping radius (given before Procedure~\ref{proc:phase}). We use the bounds in \eqref{eq:clip-summary} in the potential recursion below.

\begin{lemma}[One-step progress]\label{lem:one-step}
Under the conditions of Lemma~\ref{lem:safe-queries} at an iteration $t\in\{1,\ldots,T\}$, the following bound holds for every realization of the clipped gradient $\widehat g_t$:
\begin{equation}\label{eq:pathwise}
 \Phi_t\le \Phi_{t-1}+a^2\|\zeta_t+\beta_t\|^2-a\ip{\zeta_t+\beta_t}{z_{t-1}-x^\star}.
\end{equation}
\end{lemma}

Together, Lemmas~\ref{lem:selfbound}--\ref{lem:one-step} (proved in Appendices~\ref{app:geometry} and~\ref{app:phaseproof}) provide the query, clipping, and potential bounds needed for the phase proof. In particular, \eqref{eq:pathwise} shows what the parameters defined before Procedure~\ref{proc:phase} must control: $\Lambda$ limits clipping bias, while $T$ and the resulting $a$ control the total effect of centered noise over the phase. Appendix~\ref{app:phaseproof} shows that these choices keep $\Phi_t\le3D^2$ with high probability. Since $A_T=6D^2/\Delta$, this yields $F(y_T)\le\Delta/2$, the guarantee stated next.

\subsection{One-phase guarantee}
The preceding lemmas identify what must be controlled within a phase. We now state the resulting contraction guarantee and the number of stochastic-gradient samples it requires.

\begin{theorem}[One-phase contraction]\label{thm:phase}
Under Assumptions~\ref{ass:convexity}, \ref{ass:h01}, and~\ref{ass:oracle}, run Procedure~\ref{proc:phase} with $F(w)\le\Delta$, $w,x^\star\in Q$, and the phase parameters $\Lambda,T,a$ specified above. Then
\[
 \Prb\{F(y_T)\le\Delta/2\}\ge1-\rho.
\]
\end{theorem}

\paragraph{Proof sketch.}
Let $\tau$ be the first $t\le T$ with $\Phi_t>3D^2$ (or $T+1$ if none). Up to this time, Lemmas~\ref{lem:selfbound} and~\ref{lem:safe-queries} validate the local model, Assumption~\ref{ass:oracle} and Lemma~\ref{lem:clipping} control noise and clipping bias, and Lemma~\ref{lem:one-step} gives the potential recursion. With the chosen parameters, maximal Freedman inequalities imply that, on an event of probability at least $1-\rho$, $\Phi_\tau<3D^2$ whenever $\tau\le T$, contradicting the definition of $\tau$ \citep{freedman1975}. Hence $\Prb\{\tau\le T\}\le\rho$; on the complementary event, $\Phi_T\le3D^2$ and $A_T=6D^2/\Delta$ gives $F(y_T)\le\Delta/2$. The full proof is in Appendix~\ref{app:phaseproof}.

\section{Main Convergence Results}\label{sec:main-results}
Combining Theorem~\ref{thm:phase} with restarts gives an end-to-end guarantee for every prescribed $\varepsilon\in(0,\Delta_0]$ and $\rho\in(0,1)$: RCSAG reaches $F(w_S)\le\varepsilon$ with probability at least $1-\rho$. Clipping bias imposes no fixed accuracy floor; persistent noise increases the required number of oracle calls instead. Below, $\widetilde O_{\rho,p}(\cdot)$ suppresses constants depending on $p$ and logarithms in confidence and phase indices, but not polynomial dependence on the displayed problem parameters.
\subsection{Convex case}\label{sec:convex-case}
Fix a closed convex set $Q$ containing $x_0$ and $x^\star$, let $D:=\diam(Q)\in(0,\infty)$, and set $S:=\lceil\log_2(\Delta_0/\varepsilon)\rceil$. For $H_0,H_1>0$, define $J_H:=\min\{S,\lceil\log_2^+(H_1\Delta_0/H_0)\rceil\}$, where $\log_2^+(x):=\max\{0,\log_2x\}$; at the endpoints, set $J_H=0$ if $H_1=0$ and $J_H=S$ if $H_0=0<H_1$. Thus $J_H$ counts the phases with $H_1\Delta_s>H_0$: before the crossover, the $H_0$ term $D\sqrt{H_0/\varepsilon}$ is absorbed by the $H_1$ contribution up to a constant; afterward, $J_H$ is fixed and $H_0$ governs the later phases.

\begin{samepage}
\begin{theorem}[Convex]\label{thm:convex}
Suppose Assumptions~\ref{ass:convexity}, \ref{ass:h01}, and~\ref{ass:oracle} hold and $F(x_0)\le\Delta_0$. Run Algorithm~\ref{alg:rcsag} with $Q_s=Q$, target $\varepsilon\in(0,\Delta_0]$, and confidence $\rho\in(0,1)$. Then, with probability at least $1-\rho$, the algorithm returns $w_S$ satisfying $F(w_S)\le\varepsilon$ after $N_{\rm cvx}$ stochastic-gradient calls with $q=p/(p-1)$:
\begin{align*}
 \widetilde O_{\rho,p}\Bigg[&
 S+D\left(\sqrt{\frac{H_0}{\varepsilon}}+\sqrt{H_1}\,J_H\right)
 +\left(\frac{\sigma_0D}{\varepsilon}\right)^q+\left(\sigma_1D\left[
 \sqrt{\frac{H_0}{\varepsilon}}+\sqrt{H_1}\,J_H^{\frac{1}{q}}
 \right]\right)^q
 +\left(\frac{\sigma_2D}{\sqrt\varepsilon}\right)^q
 \Bigg].
\end{align*}
\end{theorem}
\end{samepage}
The bound in Theorem~\ref{thm:convex} can be read term by term. The term $S$ comes from the constant part of the iteration budget in each restart phase. The deterministic term $D(\sqrt{H_0/\varepsilon}+\sqrt{H_1}J_H)$ preserves accelerated square-root dependence on both curvature scales \citep{lobanov2026h01}: the $H_0$ part is optimal up to logarithms \citep{nesterov2004}, while $J_H\le S$ improves the previous $H_1$ dependence after the curvature crossover. The term $(\sigma_0D/\varepsilon)^q$ recovers the uniform finite-$p$ rate \citep{gorbunov2020heavy,nguyen2023clipped} and matches lower bounds up to logarithms (Appendix~\ref{app:lower-additive}). For gradient-dependent noise, \citet{gupta2024agnes} establish acceleration in expectation under classical smoothness for $p=2$, whereas our high-probability bound covers mixed noise for all $1<p\le2$; the optimal $\sigma_1$ dependence in this broader setting remains open. Finally, $(\sigma_2D/\sqrt\varepsilon)^q$ captures noise that decreases with the objective gap, extending the expected-smoothness/interpolation scaling \citep{gower2019sgd}, and matches lower bounds up to logarithms (Appendix~\ref{app:lower-gap}). Hence, when $\sigma_0=0$, the worst stochastic accuracy dependence improves from $\varepsilon^{-q}$ to $\varepsilon^{-q/2}$. The proof is given in Appendix~\ref{app:restart}.
\subsection{Strongly convex case}\label{sec:strongly-convex}
When $\mu>0$ (see Assumption~\ref{ass:convexity}), $F(w_s)\le\Delta_s$ implies $x^\star\in Q_s=B(w_s,\sqrt{2\Delta_s/\mu})$. Thus the auxiliary ball has diameter $D_s=2\sqrt{2\Delta_s/\mu}$ and shrinks with the gap bound. Set $S:=\lceil\log_2(\Delta_0/\varepsilon)\rceil$.

\begin{samepage}
\begin{theorem}[Strongly convex]\label{thm:strong}
Suppose Assumptions~\ref{ass:convexity}, \ref{ass:h01}, and~\ref{ass:oracle} hold with $\mu>0$ and $F(x_0)\le\Delta_0$. Run RCSAG with the balls $Q_s$ above, target $\varepsilon\in(0,\Delta_0]$, and confidence $\rho\in(0,1)$. Then, with probability at least $1-\rho$, it returns $w_S$ satisfying $F(w_S)\le\varepsilon$ after $N_{\rm sc}$ stochastic-gradient calls:
\[
\widetilde O_{\rho,p}\!\Biggl[\Bigl(1+\sqrt{\frac{H_0}{\mu}}\Bigr)S+\sqrt{\frac{H_1\Delta_0}{\mu}}+\Biggl(\frac{\sigma_0}{\sqrt{\mu\varepsilon}}\Biggr)^q+\Biggl(\frac{\sigma_1}{\sqrt\mu}\bigl[\sqrt{H_0}\,S^{1/q}+\sqrt{H_1\Delta_0}\bigr]\Biggr)^q+\Biggl(\frac{\sigma_2}{\sqrt\mu}\Biggr)^qS\Biggr].
\]
\end{theorem}
\end{samepage}
Recall that $q:=p/(p-1)$ and, as in Theorem~\ref{thm:convex}, $S$ counts the gap-halving phases. The deterministic term $(1+\sqrt{H_0/\mu})S+\sqrt{H_1\Delta_0/\mu}$ improves the nonaccelerated bound of \citet{lobanov2026bias} from linear to square-root curvature dependence; its $H_0$ dependence is optimal up to logarithms \citep{nesterov2004}. The term $(\sigma_0/\sqrt{\mu\varepsilon})^q$ attains the optimal high-probability finite-$p$ rate \citep{sadiev2023hp} and matches our lower bound up to logarithms (Appendix~\ref{app:lower-additive}). The $\sigma_1$ comparison is the same as in Theorem~\ref{thm:convex}, and its optimal dependence remains open. The term $(\sigma_2/\sqrt\mu)^qS$ recovers the logarithmic $p=2$ dependence known under expected smoothness and interpolation \citep{gower2019sgd}; the same dependence holds with high probability for all $1<p\le2$ and is sharp on a fixed-condition-number smooth subclass (Appendix~\ref{app:lower-gap}). Hence, when $\sigma_0=0$, the accuracy dependence is logarithmic even for $p<2$. The proof is in Appendix~\ref{app:restart}.
\subsection{When noise decay changes the rate}\label{sec:phase-transition}
To isolate how noise decay affects the accuracy dependence, in this subsection we keep conditional unbiasedness but replace the mixed moment bound in Assumption~\ref{ass:oracle} with
\begin{equation}\label{eq:Falpha}
 \E_t\|g(x,\xi_t)-\nabla f(x)\|^p\le\tau_\alpha^pF(x)^\alpha,\qquad \alpha\ge0.
\end{equation}
The choice $\alpha=0$ gives a uniform noise bound analogous to the $\sigma_0$ term, whereas $\alpha=p/2$ recovers the gap-dependent term $\sigma_2^pF(x)^{p/2}$ used above. Larger $\alpha$ means that the noise decreases faster as $F(x)\to0$. The next result identifies when this decay changes the dependence~on~the~desired~accuracy.

\begin{samepage}
\begin{theorem}[Power-law noise]\label{thm:alpha}
Suppose Assumptions~\ref{ass:convexity} and~\ref{ass:h01} hold, $F(x_0)\le\Delta_0$, and the oracle is conditionally unbiased and satisfies \eqref{eq:Falpha} for some $1<p\le2$. Run RCSAG with target $\varepsilon\in(0,\Delta_0]$, confidence $\rho\in(0,1)$, phase noise bound $\varsigma_{p,\Delta}^p:=\tau_\alpha^p(4\Delta)^\alpha$, and the corresponding choices of $Q_s$ specified above. Then, with probability at least $1-\rho$, it returns $w_S$ satisfying $F(w_S)\le\varepsilon$. In addition to the corresponding noise-free terms in Theorems~\ref{thm:convex} and~\ref{thm:strong},~the~noise~contributes
\[
N^{\rm cvx}_\alpha=\widetilde O_{\rho,p}\!\left((\tau_\alpha D)^q\mathcal K_p(\alpha)\right),\qquad N^{\rm sc}_\alpha=\widetilde O_{\rho,p}\!\left(\left(\frac{\tau_\alpha}{\sqrt\mu}\right)^q\mathcal K_{p/2}(\alpha)\right),
\]
where, for $\beta>0$,
\[
\mathcal K_\beta(\alpha):=\begin{cases}
\varepsilon^{-(\beta-\alpha)/(p-1)},&\alpha<\beta,\\
\log(\Delta_0/\varepsilon),&\alpha=\beta,\\
\Delta_0^{(\alpha-\beta)/(p-1)},&\alpha>\beta.
\end{cases}
\]
\end{theorem}
Theorem~\ref{thm:alpha} separates tail heaviness, controlled by $p$, from noise decay, controlled by $\alpha$. In a convex phase, the stochastic cost scales as $(\tau_\alpha D)^q\Delta_s^{(\alpha-p)/(p-1)}$, so $\alpha=p$ separates polynomial, logarithmic, and accuracy-independent costs. Strong convexity gives $D_s=2\sqrt{2\Delta_s/\mu}$ and shifts the threshold to $\alpha=p/2$; hence the power used in Assumption~\ref{ass:oracle} is polynomial in the convex case but logarithmic in the strongly convex case. To our knowledge, this is the first high-probability accelerated result to identify both finite-$p$ transitions. The proof is in Appendix~\ref{app:structural}.
\end{samepage}

\section{Discussion and Limitations}\label{sec:limits}
RCSAG requires an initial bound $F(x_0)\le\Delta_0$ but never evaluates the gap during the run; Appendix~\ref{app:init} constructs such a bound from stochastic gradients without knowing $f^\star$ or an exact gradient. In the convex case, the method also needs a bounded set $Q$ containing $x_0$ and an optimizer, together with an efficient projection onto $Q$; standard constructions are given in Appendix~\ref{app:choosing-Q}. Strong convexity replaces this requirement by shrinking balls determined by $\Delta_s$ and $\mu$, which are explicit in common regularized models. The tuned Procedure~\ref{proc:phase} also uses $p$ and upper bounds on $\sigma_i$ and $H_i$. The horizon-driven variant in Appendix~\ref{app:parameterfree} removes them from the inner updates for a prescribed $T$, but the unknown parameters still determine a sufficient horizon, which the method cannot infer from observed gradients; the strongly convex wrapper also needs $\mu$. Two theoretical questions remain open: the optimal dependence on $\sigma_1$ and matching lower bounds for the general $F(x)^\alpha$ noise model (Appendices~\ref{app:sigma1} and~\ref{app:structural}). Finally, the experiments use practical tuning and a convex head on frozen features; the finite dataset is not a literal infinite-variance oracle. Extending the theory and experiments to end-to-end nonconvex training is left for future work.

\vspace{-0.5em}
\section{Experiments}\label{sec:experiments}
\vspace{-0.5em}
Linear probing trains a prediction head while keeping the pretrained encoder fixed, and is routinely used to evaluate or adapt learned representations \citep{kornblith2019better,kumar2022finetuning}. It gives us a convex problem built from real image features, without the additional effects of updating the encoder. We use this setup to ask a simple question: does RCSAG preserve acceleration when some examples produce unusually large gradients?
\vspace{-1.3em}

\begin{wrapfigure}{r}{0.52\textwidth}
\vspace{-1.2em}
\centering
\includegraphics[width=\linewidth]{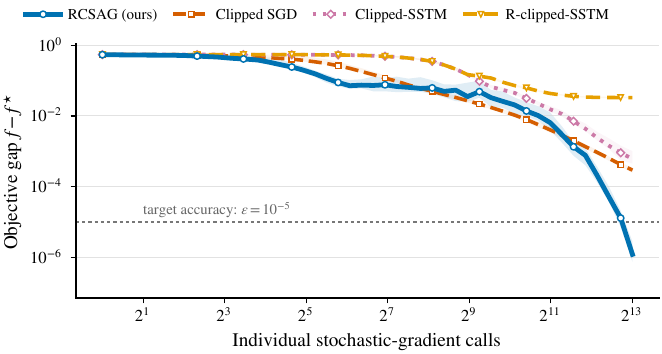}
\vspace{-0.8em}
\caption{Median gap and 10th--90th percentile bands over 30 seeds with fixed Pareto weights.}
\label{fig:real-cosh}
\vspace{-1.0em}
\end{wrapfigure}

\paragraph{Setup and protocol.}
We train a kernelized linear head on frozen DINOv2 ViT-S/14 features \citep{oquab2024dinov2} from 128 balanced CIFAR-10 images \citep{krizhevsky2009learning}. For the normalized regularized Gram matrix $K$ with $\kappa(K)=20$, we minimize $f(\alpha)=n^{-1}\sum_i\omega_i[\cosh((K\alpha)_i-y_i)-1]$. For the realized $K$ and weights, the problem is $\mu$-strongly convex with $\mu=7.65\times10^{-6}$ and satisfies the Hessian $(H_0,H_1)$ bound with $H_0=H_1=R^2=2.77\times10^{-2}$, where $R=\max_i\|K_{i:}\|$. It also exactly interpolates: all component gradients vanish when $K\alpha^\star=y$. The fixed weights $\omega_i$ are drawn once from a Pareto distribution with shape $1.5$ and create occasional large component gradients without adding oracle noise. We~compare~RCSAG (Algorithm~\ref{alg:rcsag}) with Clipped SGD \citep{gorbunov2020heavy}, Clipped-SSTM \citep{sadiev2023hp}, and a practical restarted Clipped-SSTM variant \citep{gorbunov2020heavy}. All methods start from zero, use the same sampled indices and one component gradient per iteration, and receive $8{,}192$ calls. We tune 16 configurations per method on six seeds and evaluate on 30 disjoint seeds. Appendix~\ref{app:experiments} gives the grids, an equal-weight control, and~an~additional~synthetic~experiment.
\vspace{-1.3em}
\paragraph{Results.}
RCSAG is the only method whose final gap is below $\varepsilon=10^{-5}$ on every held-out seed. Its median gap is $1.22\times10^{-6}$, compared with $2.95\times10^{-4}$ for Clipped SGD, the next-best method, while only $0.15\%$ of its updates are clipped. Thus rare clipping is enough to preserve fast progress on this problem; we do not claim uniform empirical superiority.

\vspace{-0.5em}
\section{Conclusion}
\vspace{-1em}
This work answers the central question affirmatively: clipping can preserve acceleration under nonuniform curvature and mixed heavy-tailed noise. RCSAG turns a gap bound into phasewise scales; projection keeps momentum queries inside the local model, while single-sample clipping and a stopped potential control errors and bias. Restarts reach any prescribed accuracy with high probability, without an error floor, while retaining square-root dependence on both curvature scales. Under strong convexity, accuracy dependence is logarithmic when $\sigma_0=0$, even for $p<2$; the $F(x)^\alpha$ model exposes transitions at $\alpha=p$ and $\alpha=p/2$. Together, these results establish phasewise localization as a general design principle for accelerated optimization with evolving curvature and heavy-tailed noise without reducing either source to a uniform global worst-case bound.

\section*{AI Use Statement}
\addcontentsline{toc}{section}{AI Use Statement}
Generative AI tools were used for language editing and to improve the clarity of the manuscript. We reviewed all resulting changes and take full responsibility for the final content.

\section*{Acknowledgments}
We thank Yuriy Dorn for his support.


\bibliographystyle{unsrtnat}
\bibliography{3_references}  

@book{nesterov2004,
  title={Introductory Lectures on Convex Optimization: A Basic Course},
  author={Nesterov, Yurii},
  year={2004},
  publisher={Kluwer Academic Publishers}
}

@techreport{krizhevsky2009learning,
  title={Learning Multiple Layers of Features from Tiny Images},
  author={Krizhevsky, Alex},
  institution={University of Toronto},
  year={2009}
}

@inproceedings{kornblith2019better,
  title={Do Better {ImageNet} Models Transfer Better?},
  author={Kornblith, Simon and Shlens, Jonathon and Le, Quoc V.},
  booktitle={Proceedings of the IEEE/CVF Conference on Computer Vision and Pattern Recognition},
  pages={2661--2671},
  year={2019}
}

@inproceedings{kumar2022finetuning,
  title={Fine-Tuning Can Distort Pretrained Features and Underperform Out-of-Distribution},
  author={Kumar, Ananya and Raghunathan, Aditi and Jones, Robbie and Ma, Tengyu and Liang, Percy},
  booktitle={International Conference on Learning Representations},
  year={2022}
}

@article{oquab2024dinov2,
  title={{DINOv2}: Learning Robust Visual Features without Supervision},
  author={Oquab, Maxime and Darcet, Timoth{\'e}e and Moutakanni, Th{\'e}o and Vo, Huy V. and Szafraniec, Marc and Khalidov, Vasil and Fernandez, Pierre and Haziza, Daniel and Massa, Francisco and El-Nouby, Alaaeldin and Assran, Mido and Ballas, Nicolas and Galuba, Wojciech and Howes, Russell and Huang, Po-Yao and Li, Shang-Wen and Misra, Ishan and Rabbat, Michael and Sharma, Vasu and Synnaeve, Gabriel and Xu, Hu and J{\'e}gou, Herv{\'e} and Mairal, Julien and Labatut, Patrick and Joulin, Armand and Bojanowski, Piotr},
  journal={Transactions on Machine Learning Research},
  year={2024}
}

@article{ghadimi2013shrinking,
  title={Optimal Stochastic Approximation Algorithms for Strongly Convex Stochastic Composite Optimization, {II}: Shrinking Procedures and Optimal Algorithms},
  author={Ghadimi, Saeed and Lan, Guanghui},
  journal={SIAM Journal on Optimization},
  volume={23},
  number={4},
  pages={2061--2089},
  year={2013},
  doi={10.1137/110848876}
}

@inproceedings{gorbunov2020heavy,
  title={Stochastic Optimization with Heavy-Tailed Noise via Accelerated Gradient Clipping},
  author={Gorbunov, Eduard and Danilova, Marina and Gasnikov, Alexander},
  booktitle={Advances in Neural Information Processing Systems},
  volume={33},
  year={2020}
}

@inproceedings{nguyen2023clipped,
  title={Improved Convergence in High Probability of Clipped Gradient Methods with Heavy Tailed Noise},
  author={Nguyen, Ta Duy and Nguyen, Thien H. and Ene, Alina and Nguyen, Huy},
  booktitle={Advances in Neural Information Processing Systems},
  volume={36},
  year={2023}
}

@inproceedings{sadiev2023hp,
  title={High-Probability Bounds for Stochastic Optimization and Variational Inequalities: The Case of Unbounded Variance},
  author={Sadiev, Abdurakhmon and Danilova, Marina and Gorbunov, Eduard and Horvath, Samuel and Gidel, Gauthier and Dvurechensky, Pavel and Gasnikov, Alexander and Richtarik, Peter},
  booktitle={Proceedings of the 40th International Conference on Machine Learning},
  series={Proceedings of Machine Learning Research},
  volume={202},
  pages={29563--29648},
  year={2023}
}

@inproceedings{gower2019sgd,
  title={SGD: General Analysis and Improved Rates},
  author={Gower, Robert Mansel and Loizou, Nicolas and Qian, Xun and Sailanbayev, Alibek and Shulgin, Egor and Richtarik, Peter},
  booktitle={Proceedings of the 36th International Conference on Machine Learning},
  series={Proceedings of Machine Learning Research},
  volume={97},
  pages={5200--5209},
  year={2019}
}

@article{khaled2023better,
  title={Better Theory for SGD in the Nonconvex World},
  author={Khaled, Ahmed and Richtarik, Peter},
  journal={Transactions on Machine Learning Research},
  year={2023}
}

@inproceedings{zhang2020clipping,
  title={Why Gradient Clipping Accelerates Training: A Theoretical Justification for Adaptivity},
  author={Zhang, Jingzhao and He, Tianxing and Sra, Suvrit and Jadbabaie, Ali},
  booktitle={International Conference on Learning Representations},
  year={2020}
}

@inproceedings{li2023generalized,
  title={Convex and Non-convex Optimization Under Generalized Smoothness},
  author={Li, Haochuan and Qian, Jian and Tian, Yi and Rakhlin, Alexander and Jadbabaie, Ali},
  booktitle={Advances in Neural Information Processing Systems},
  volume={36},
  year={2023}
}

@article{lobanov2026h01,
  title={A Few Accelerated Algorithms for Convex Optimization under {$(H_0,H_1)$}-Smoothness},
  author={Lobanov, Aleksandr},
  journal={arXiv preprint arXiv:2608.04884},
  year={2026}
}

@article{lobanov2024linear,
  title={Linear Convergence Rate in Convex Setup is Possible! Gradient Descent Method Variants under {$(L_0,L_1)$}-Smoothness},
  author={Lobanov, Aleksandr and Gasnikov, Alexander and Gorbunov, Eduard and Tak{\'a}{\v c}, Martin},
  journal={arXiv preprint arXiv:2412.17050},
  year={2024}
}

@article{lobanov2025power,
  title={Power of {$(L_0,L_1)$}-Smoothness in Stochastic Convex Optimization: First- and Zero-Order Algorithms},
  author={Lobanov, Aleksandr and Gasnikov, Alexander},
  journal={arXiv preprint arXiv:2501.18198},
  year={2025}
}

@article{dvinskikh2026localize,
  title={Localize, Restart, Accelerate: Stochastic Optimization under Generalized Smoothness},
  author={Dvinskikh, Darina and Gasnikov, Alexander and Lobanov, Aleksandr and Latypov, Ilgam},
  journal={arXiv preprint arXiv:2609.06555},
  year={2026}
}

@article{yuhonglin2025rsag,
  title={Convergence Analysis of Stochastic Accelerated Gradient Methods for Generalized Smooth Optimizations},
  author={Yu, Chenhao and Hong, Yusu and Lin, Junhong},
  journal={arXiv preprint arXiv:2502.11125},
  year={2025}
}

@inproceedings{gupta2024agnes,
  title={Nesterov Acceleration Despite Very Noisy Gradients},
  author={Gupta, Kanan and Siegel, Jonathan W. and Wojtowytsch, Stephan},
  booktitle={Advances in Neural Information Processing Systems},
  volume={37},
  year={2024}
}

@article{vasin2023relative,
  title={Accelerated Gradient Methods with Absolute and Relative Noise in the Gradient},
  author={Vasin, Artem and Gasnikov, Alexander and Dvurechensky, Pavel and Spokoiny, Vladimir},
  journal={Optimization Methods and Software},
  volume={38},
  number={6},
  pages={1180--1229},
  year={2023},
  doi={10.1080/10556788.2023.2212503}
}

@article{kornilov2025intermediate,
  title={Intermediate Gradient Methods with Relative Inexactness},
  author={Kornilov, Nikita and Alkousa, Mohammad and Gorbunov, Eduard and Stonyakin, Fedor and Dvurechensky, Pavel and Gasnikov, Alexander},
  journal={Journal of Optimization Theory and Applications},
  volume={207},
  pages={62},
  year={2025},
  doi={10.1007/s10957-025-02809-y}
}

@inproceedings{faw2022adaptivity,
  title={The Power of Adaptivity in SGD: Self-Tuning Step Sizes with Unbounded Gradients and Affine Variance},
  author={Faw, Matthew and Tziotis, Isidoros and Caramanis, Constantine and Mokhtari, Aryan and Shakkottai, Sanjay and Ward, Rachel},
  booktitle={Proceedings of Thirty Fifth Conference on Learning Theory},
  series={Proceedings of Machine Learning Research},
  volume={178},
  pages={313--355},
  year={2022}
}

@inproceedings{hong2024adagrad,
  title={Revisiting Convergence of AdaGrad with Relaxed Assumptions},
  author={Hong, Yusu and Lin, Junhong},
  booktitle={Proceedings of the Fortieth Conference on Uncertainty in Artificial Intelligence},
  series={Proceedings of Machine Learning Research},
  volume={244},
  pages={1727--1750},
  year={2024}
}

@article{freedman1975,
  title={On Tail Probabilities for Martingales},
  author={Freedman, David A.},
  journal={The Annals of Probability},
  volume={3},
  number={1},
  pages={100--118},
  year={1975}
}

@article{minsker2015geom,
  title={Geometric Median and Robust Estimation in Banach Spaces},
  author={Minsker, Stanislav},
  journal={Bernoulli},
  volume={21},
  number={4},
  pages={2308--2335},
  year={2015}
}

@inproceedings{pascanu2013rnn,
  title={On the Difficulty of Training Recurrent Neural Networks},
  author={Pascanu, Razvan and Mikolov, Tomas and Bengio, Yoshua},
  booktitle={Proceedings of the 30th International Conference on Machine Learning},
  series={Proceedings of Machine Learning Research},
  volume={28},
  pages={1310--1318},
  year={2013}
}

@article{chowdhery2023palm,
  title={{PaLM}: Scaling Language Modeling with Pathways},
  author={Chowdhery, Aakanksha and others},
  journal={Journal of Machine Learning Research},
  volume={24},
  number={240},
  pages={1--113},
  year={2023}
}

@inproceedings{koloskova2023clipping,
  title={Revisiting Gradient Clipping: Stochastic Bias and Tight Convergence Guarantees},
  author={Koloskova, Anastasia and Hendrikx, Hadrien and Stich, Sebastian U.},
  booktitle={Proceedings of the 40th International Conference on Machine Learning},
  series={Proceedings of Machine Learning Research},
  volume={202},
  pages={17343--17363},
  year={2023}
}

@inproceedings{simsekli2019tail,
  title={A Tail-Index Analysis of Stochastic Gradient Noise in Deep Neural Networks},
  author={Simsekli, Umut and Sagun, Levent and Gurbuzbalaban, Mert},
  booktitle={Proceedings of the 36th International Conference on Machine Learning},
  series={Proceedings of Machine Learning Research},
  volume={97},
  pages={5827--5837},
  year={2019}
}

@inproceedings{battash2024noise,
  title={Revisiting the Noise Model of Stochastic Gradient Descent},
  author={Battash, Barak and Wolf, Lior and Lindenbaum, Ofir},
  booktitle={Proceedings of the 27th International Conference on Artificial Intelligence and Statistics},
  series={Proceedings of Machine Learning Research},
  volume={238},
  pages={4780--4788},
  year={2024}
}

@inproceedings{ma2018interpolation,
  title={The Power of Interpolation: Understanding the Effectiveness of {SGD} in Modern Over-Parametrized Learning},
  author={Ma, Siyuan and Bassily, Raef and Belkin, Mikhail},
  booktitle={Proceedings of the 35th International Conference on Machine Learning},
  series={Proceedings of Machine Learning Research},
  volume={80},
  pages={3325--3334},
  year={2018}
}

@article{liu2025warmup,
  title={Theoretical Analysis on How Learning Rate Warmup Accelerates Convergence},
  author={Liu, Yuxing and Ge, Yuze and Pan, Rui and Kang, An and Zhang, Tong},
  journal={arXiv preprint arXiv:2509.07972},
  year={2025}
}

@article{alimisis2025warmup,
  title={Why Do We Need Warm-up? A Theoretical Perspective},
  author={Alimisis, Foivos and Islamov, Rustem and Lucchi, Aurelien},
  journal={arXiv preprint arXiv:2510.03164},
  year={2025}
}

@inproceedings{vaswani2025armijo,
  title={Armijo Line-search Can Make ({S}tochastic) Gradient Descent Provably Faster},
  author={Vaswani, Sharan and Babanezhad Harikandeh, Reza},
  booktitle={Proceedings of the 42nd International Conference on Machine Learning},
  series={Proceedings of Machine Learning Research},
  volume={267},
  pages={61035--61070},
  year={2025}
}

@article{vaswani2026nonuniform,
  title={Convergence of Steepest Descent and {Adam} under Non-Uniform Smoothness},
  author={Vaswani, Sharan and Sun, Yifan and Babanezhad, Reza},
  journal={arXiv preprint arXiv:2605.30648},
  year={2026}
}

@article{zhanglin2026scale,
  title={Scale-Invariant Neural Network Optimization: Norm Geometry and Heavy-Tailed Noise},
  author={Zhang, Jiayu and Lin, Tianyi},
  journal={arXiv preprint arXiv:2605.18528},
  year={2026}
}

@article{lobanov2026bias,
  title={Avoiding Bias in Clipped {SGD} for Overparameterized Models under Generalized Smoothness},
  author={Lobanov, Aleksandr and Koloskova, Anastasia},
  journal={arXiv preprint arXiv:2605.14800},
  year={2026}
}

\appendix
\newcommand{\proofparagraph}[1]{%
  \par\smallskip\noindent\textbf{#1}\nobreak\enspace}
\AtBeginEnvironment{proof}{\let\paragraph\proofparagraph}

\section{Auxiliary Results for One Phase}\label{app:geometry}
This appendix proves Lemmas~\ref{lem:selfbound}--\ref{lem:one-step}. The proofs follow the order in which the results are used by Procedure~\ref{proc:phase}. Lemma~\ref{lem:selfbound} first converts an objective-gap bound into a gradient bound. Lemma~\ref{lem:safe-queries} then shows that the accelerated query and averaging steps remain inside the radius where Assumption~\ref{ass:h01} is valid. Lemma~\ref{lem:clipping} controls the centered error and the bias created by clipping. Finally, Lemma~\ref{lem:one-step} combines these ingredients in the potential recursion used to prove the one-phase contraction.

Throughout this appendix, recall that $\mathcal H(x):=H_0+H_1F(x)$, $\mathcal H_\Delta:=H_0+H_1\Delta$, and $C_\star=19+12\sqrt{2}$ when $H_1>0$, whereas $C_\star=2$ when $H_1=0$.

\subsection{Proof of Lemma~\ref{lem:selfbound}}
\begin{proof}
Fix an arbitrary $x\in\R^d$ and write
$F(x)=f(x)-f^\star$ and
$\mathcal H(x)=H_0+H_1F(x)$. We first use one step of exactly the local-model radius to prove the estimate involving $r_H^{-1}$. We then use a shorter quadratic-model step to convert this estimate into the gradient self-bound. If $\nabla f(x)=0$, both conclusions hold because their left-hand sides vanish, so below we assume $\nabla f(x)\ne0$.

\paragraph{Estimate at the local-model radius.}
Suppose first that $H_1>0$. Define
\[
 y:=x-r_H\frac{\nabla f(x)}{\|\nabla f(x)\|}.
\]
The definition gives $\|y-x\|=r_H$, so Assumption~\ref{ass:h01} applies with base point $x$ and trial point $y$. Moreover, $f^\star\le f(y)$ because $f^\star$ is the minimum value. Consequently,
\begin{align}
 f^\star
 &\le f(y)\le f(x)+\ip{\nabla f(x)}{y-x}+\mathcal H(x)\|y-x\|^2=f(x)-r_H\|\nabla f(x)\|+\mathcal H(x)r_H^2.
 \label{eq:app-radius-trial}
\end{align}
Subtracting $f^\star$ from the last line of \eqref{eq:app-radius-trial}, rearranging, and dividing by $r_H^2>0$ yield
\begin{align*}
 \frac{\|\nabla f(x)\|}{r_H}
 &\le \frac{F(x)}{r_H^2}+\mathcal H(x)=\bigl(18+12\sqrt{2}\bigr)H_1F(x)+\mathcal H(x)\le\bigl(19+12\sqrt{2}\bigr)\mathcal H(x)
 =C_\star\mathcal H(x).
\end{align*}
The equality uses $r_H^{-2}=(2\sqrt{3}+\sqrt{6})^2H_1=(18+12\sqrt{2})H_1$. The last inequality follows from $H_1F(x)\le H_0+H_1F(x)=\mathcal H(x)$. If $H_1=0$, then $r_H=\infty$ and the convention $r_H^{-1}=0$ makes the first inequality in \eqref{eq:selfbounds} automatic.

\paragraph{Gradient self-bound.}
We next prove the second inequality. If $\mathcal H(x)=0$, then $\nabla f(x)=0$: for $H_1>0$, the equality $\mathcal H(x)=H_0+H_1F(x)=0$ implies $F(x)=0$, so $x$ is a minimizer of the differentiable convex function; for $H_1=0$, Assumption~\ref{ass:h01} has zero quadratic term globally, and the existence of a minimizer again forces the gradient to vanish. We may therefore assume that $\mathcal H(x)>0$ and define
\[
 M_x:=C_\star\mathcal H(x),
 \qquad
 x^+:=x-\frac{\nabla f(x)}{M_x}.
\]
When $H_1>0$, the first inequality in \eqref{eq:selfbounds} gives $\|x^+-x\|=\|\nabla f(x)\|/M_x\le r_H$. When $H_1=0$, the upper model is global. Thus Assumption~\ref{ass:h01} is valid for the pair $(x,x^+)$. Since $C_\star\ge2$, we also have $M_x\ge2\mathcal H(x)$, and hence
\begin{align}
 f(x^+)
 &\le f(x)-\frac{\|\nabla f(x)\|^2}{M_x}
 +\mathcal H(x)\frac{\|\nabla f(x)\|^2}{M_x^2}\le f(x)-\frac{\|\nabla f(x)\|^2}{2M_x}.
 \label{eq:app-self-descent}
\end{align}
The last inequality follows from $\mathcal H(x)/M_x\le1/2$. Using $f^\star\le f(x^+)$ in \eqref{eq:app-self-descent} gives
\[
 F(x)=f(x)-f^\star
 \ge \frac{\|\nabla f(x)\|^2}{2M_x}
 =\frac{\|\nabla f(x)\|^2}{2C_\star\mathcal H(x)}.
\]
Multiplying by $2C_\star\mathcal H(x)$ proves the second inequality in \eqref{eq:selfbounds}.
\end{proof}

\subsection{Proof of Lemma~\ref{lem:safe-queries}}
\begin{proof}
Fix $t\in\{1,\ldots,T\}$ and assume
$\Phi_{t-1}\le3D^2$. We first verify that the parameter choice makes each accelerated displacement short enough for Assumption~\ref{ass:h01}. We then use the potential bound to control $F(y_{t-1})$, transfer this control to the query $u_t$, and finally derive the gradient bound at $u_t$.

\paragraph{Weight control.}
Define
\[
 M:=32C_\star\mathcal H_\Delta,
\]
so that $G_\Delta=\sqrt{\Delta M}$. The clipping radius satisfies $\Lambda\ge2G_\Delta$, and therefore
\[
 R=\frac{D\Lambda}{\Delta}
 \ge2D\sqrt{\frac{M}{\Delta}}.
\]
Because $\rho\in(0,1)$ and $T=\lceil1024(1+R)\log(8/\rho)\rceil$, we have $T\ge(5/2)R\ge5D\sqrt{M/\Delta}$. Substituting $a=5D^2/(\Delta T)$ and $A_0=D^2/\Delta$ then gives
\begin{equation}
 Ma^2=\frac{25MD^4}{\Delta^2T^2}\le\frac{D^2}{\Delta}=A_0\le A_t.
 \label{eq:app-weight-control}
\end{equation}

\paragraph{Previous output gap.}
Since
$\Phi_{t-1}=A_{t-1}F(y_{t-1})+\frac12\|z_{t-1}-x^\star\|^2$ and the squared-distance term is nonnegative,
\[
 A_{t-1}F(y_{t-1})\le\Phi_{t-1}\le3D^2.
\]
Moreover, $A_{t-1}\ge A_0=D^2/\Delta$. Dividing by $A_{t-1}>0$ proves
\begin{equation}
 F(y_{t-1})\le3\Delta.
 \label{eq:app-output-gap}
\end{equation}

\paragraph{First local displacement.}
Let $\lambda_t:=a/A_t$. Equation~\ref{eq:app-weight-control} implies
\[
 \lambda_t^2=\frac{a^2}{A_t^2}\le\frac{1}{MA_t}\le\frac{1}{MA_0}.
\]
The set $Q$ is convex, $y_0=z_0=w\in Q$, projection gives $z_j\in Q$, and the update of $y_j$ is a convex combination of two points in $Q$. Hence $y_j,z_j\in Q$ for every $j$. Using the definition of $u_t$ and the diameter bound for $Q$, we obtain
\begin{align}
 \|u_t-y_{t-1}\|
 &=\lambda_t\|z_{t-1}-y_{t-1}\|\nonumber\\
 &\le\lambda_tD
 \le\sqrt{\frac{\Delta}{M}}.
 \label{eq:app-mix-step}
\end{align}
If $H_1>0$, then $M\ge32C_\star H_1\Delta$, and thus
\[
 \frac{\Delta}{M}\le\frac{1}{32C_\star H_1}\le r_H^2.
\]
The last inequality follows from $r_H^{-2}=(18+12\sqrt{2})H_1$ and $32C_\star\ge18+12\sqrt{2}$. Therefore \eqref{eq:app-mix-step} gives $\|u_t-y_{t-1}\|\le r_H$. For $H_1=0$, the same conclusion follows from $r_H=\infty$.

\paragraph{Query gap and curvature.}
We may now apply the local upper model from $y_{t-1}$ to $u_t$ because \eqref{eq:app-mix-step} verifies its radius condition. First, \eqref{eq:app-output-gap} gives
\[
 \mathcal H(y_{t-1})=H_0+H_1F(y_{t-1})
 \le H_0+3H_1\Delta\le3\mathcal H_\Delta.
\]
Combining this inequality with the second bound in Lemma~\ref{lem:selfbound} yields
\[
 \|\nabla f(y_{t-1})\|^2
 \le2C_\star F(y_{t-1})\mathcal H(y_{t-1})
 \le18C_\star\Delta\mathcal H_\Delta.
\]
Assumption~\ref{ass:h01}, the Cauchy--Schwarz inequality, and \eqref{eq:app-mix-step} therefore give
\begin{align*}
 F(u_t)
 &\le F(y_{t-1})
 +\|\nabla f(y_{t-1})\|\,\|u_t-y_{t-1}\|
 +\mathcal H(y_{t-1})\|u_t-y_{t-1}\|^2\\
 &\le3\Delta
 +\sqrt{18C_\star\Delta\mathcal H_\Delta}
   \sqrt{\frac{\Delta}{32C_\star\mathcal H_\Delta}}
 +3\mathcal H_\Delta\frac{\Delta}{32C_\star\mathcal H_\Delta}\\
 &=\left(3+\frac34+\frac{3}{32C_\star}\right)\Delta
 <4\Delta.
\end{align*}
This proves the second gap bound in \eqref{eq:safe-main}. It also implies
\begin{equation}
 \mathcal H(u_t)=H_0+H_1F(u_t)\le H_0+4H_1\Delta\le4\mathcal H_\Delta.
 \label{eq:app-query-curvature}
\end{equation}

\paragraph{Second local displacement and query gradient.}
Next, subtracting the definitions of $y_t$ and $u_t$ gives
\[
 y_t-u_t=\lambda_t(z_t-z_{t-1}).
\]
Both $z_t$ and $z_{t-1}$ belong to $Q$, so the same argument as in \eqref{eq:app-mix-step} yields
$\|y_t-u_t\|\le\lambda_tD\le r_H$. Thus both displacements used later in the local-model argument are valid.

Finally, \eqref{eq:app-query-curvature} and the definition of $M$ give
\begin{equation}
 \frac{\mathcal H(u_t)}{M}
 \le\frac{4\mathcal H_\Delta}{32C_\star\mathcal H_\Delta}
 =\frac{1}{8C_\star}\le\frac1{16}.
 \label{eq:app-curvature-ratio}
\end{equation}
Applying the second inequality in Lemma~\ref{lem:selfbound} at $u_t$ and using $F(u_t)\le4\Delta$ together with \eqref{eq:app-query-curvature}, we obtain
\[
 \|\nabla f(u_t)\|^2
 \le2C_\star F(u_t)\mathcal H(u_t)
 \le32C_\star\Delta\mathcal H_\Delta
 =G_\Delta^2.
\]
Taking square roots proves $\|\nabla f(u_t)\|\le G_\Delta$ and completes the proof of \eqref{eq:safe-main}.
\end{proof}

\subsection{Proof of Lemma~\ref{lem:clipping}}
\begin{proof}
Fix an iteration $t$ on which the localization conclusions of
Lemma~\ref{lem:safe-queries} hold. We condition on the past, first derive the localized $p$th-moment bound, then verify the almost-sure clipping envelope, the clipping bias, and the conditional second moment of the centered error.

\paragraph{Conditioning and local moment bound.}
Condition on $\mathcal F_{t-1}$. The query $u_t$ is then fixed because it is determined by the history before the sample $\xi_t$ is drawn. Lemma~\ref{lem:safe-queries} and the assumption $\Lambda\ge2G_\Delta$ give
\[
 \|\nabla f(u_t)\|\le G_\Delta\le\frac{\Lambda}{2}.
\]
Thus $\nabla f(u_t)$ lies in the Euclidean ball of radius $\Lambda$. Assumption~\ref{ass:oracle}, together with $F(u_t)\le4\Delta$ and the displayed gradient bound, gives
\begin{equation}
 \E_t\|g(u_t,\xi_t)-\nabla f(u_t)\|^p
 \le\sigma_0^p+\sigma_1^pG_\Delta^p+\sigma_2^p(4\Delta)^{p/2}
 =\varsigma_{p,\Delta}^p.
 \label{eq:app-phase-noise}
\end{equation}

\paragraph{Almost-sure bound.}
Both $\widehat g_t=\clip_\Lambda(g(u_t,\xi_t))$ and its conditional expectation $\E_t[\widehat g_t]$ belong to the radius-$\Lambda$ ball: the first statement follows from clipping, and the second follows from convexity of the ball. Therefore
\[
 \|\zeta_t\|=\|\widehat g_t-\E_t[\widehat g_t]\|
 \le\|\widehat g_t\|+\|\E_t[\widehat g_t]\|\le2\Lambda
 \qquad\text{almost surely}.
\]

\paragraph{Clipping bias.}
Conditional unbiasedness in Assumption~\ref{ass:oracle} gives $\E_t[g(u_t,\xi_t)]=\nabla f(u_t)$, and hence
\[
 \beta_t
 =\E_t[\widehat g_t]-\nabla f(u_t)
 =\E_t[\widehat g_t-g(u_t,\xi_t)].
\]
The vectors $\widehat g_t$ and $g(u_t,\xi_t)$ differ only on the event $\{\|g(u_t,\xi_t)\|>\Lambda\}$. On this event,
\[
 \|g(u_t,\xi_t)-\nabla f(u_t)\|
 \ge\|g(u_t,\xi_t)\|-\|\nabla f(u_t)\|
 >\frac{\Lambda}{2}.
\]
Moreover, clipping is Euclidean projection onto the radius-$\Lambda$ ball, and $\nabla f(u_t)$ belongs to that ball. The defining optimality of projection therefore gives
\[
 \|\widehat g_t-g(u_t,\xi_t)\|
 \le\|g(u_t,\xi_t)-\nabla f(u_t)\|.
\]
Combining these facts with conditional Jensen's inequality yields
\begin{align*}
 \|\beta_t\|
 &\le\E_t\!\left[\|g(u_t,\xi_t)-\nabla f(u_t)\|
 \mathbf1\{\|g(u_t,\xi_t)-\nabla f(u_t)\|>\Lambda/2\}\right]\\
 &\le(\Lambda/2)^{1-p}
 \E_t\|g(u_t,\xi_t)-\nabla f(u_t)\|^p\\
 &\le2\varsigma_{p,\Delta}^p\Lambda^{1-p}.
\end{align*}
The second inequality uses $s\mathbf1\{s>c\}\le c^{1-p}s^p$ for $s,c>0$, and the last uses \eqref{eq:app-phase-noise} together with $2^{p-1}\le2$ for $1<p\le2$.

\paragraph{Conditional second moment.}
It remains to control the conditional second moment of the centered term. Conditional expectation minimizes mean squared distance, so
\[
 \E_t\|\zeta_t\|^2
 =\E_t\|\widehat g_t-\E_t[\widehat g_t]\|^2
 \le\E_t\|\widehat g_t-\nabla f(u_t)\|^2.
\]
Projection onto a closed convex set is nonexpansive, and $\nabla f(u_t)$ lies in the clipping ball. Hence
$\|\widehat g_t-\nabla f(u_t)\|\le\|g(u_t,\xi_t)-\nabla f(u_t)\|$. At the same time,
$\|\widehat g_t-\nabla f(u_t)\|\le\|\widehat g_t\|+\|\nabla f(u_t)\|\le2\Lambda$. It follows that
\[
 \|\widehat g_t-\nabla f(u_t)\|^2
 \le(2\Lambda)^{2-p}\|g(u_t,\xi_t)-\nabla f(u_t)\|^p.
\]
Taking conditional expectations, using \eqref{eq:app-phase-noise}, and observing that $2^{2-p}\le2$ give
\[
 \E_t\|\zeta_t\|^2
 \le2^{2-p}\varsigma_{p,\Delta}^p\Lambda^{2-p}
 \le2\varsigma_{p,\Delta}^p\Lambda^{2-p}.
\]
This proves all three inequalities in \eqref{eq:clip-summary}.
\end{proof}

\subsection{Proof of Lemma~\ref{lem:one-step}}
\begin{proof}
Fix an iteration $t$ satisfying the premise of Lemma~\ref{lem:safe-queries}. We first apply the local upper model to the averaging update, then use convexity to couple the objective terms, and finally use the Euclidean projection optimality condition to control the auxiliary step.
\paragraph{Local-model progress.}
Lemma~\ref{lem:safe-queries} gives $\|y_t-u_t\|\le r_H$, so Assumption~\ref{ass:h01} applies from $u_t$ to $y_t$. Subtracting the definitions of $y_t$ and $u_t$ gives
\[
 y_t-u_t=\frac{a}{A_t}(z_t-z_{t-1}).
\]
Multiplying the resulting upper-model inequality by $A_t$ therefore yields
\begin{align}
 A_tF(y_t)
 \le{}&A_tF(u_t)
 +a\ip{\nabla f(u_t)}{z_t-z_{t-1}}+\frac{\mathcal H(u_t)a^2}{A_t}\|z_t-z_{t-1}\|^2.
 \label{eq:app-local-progress}
\end{align}
Equation~\ref{eq:app-weight-control} gives $Ma^2\le A_t$, while \eqref{eq:app-curvature-ratio} gives $\mathcal H(u_t)/M\le1/16$. Thus the coefficient of the squared displacement in \eqref{eq:app-local-progress} satisfies
\begin{equation}
 \frac{\mathcal H(u_t)a^2}{A_t}\le\frac{\mathcal H(u_t)}{M}\le\frac1{16}.
 \label{eq:app-quadratic-coefficient}
\end{equation}

\paragraph{Convex coupling.}
We next relate $F(u_t)$ to the previous output gap. Convexity at $u_t$ gives
\[
 F(u_t)-F(y_{t-1})\le\ip{\nabla f(u_t)}{u_t-y_{t-1}},
 \qquad
 F(u_t)\le\ip{\nabla f(u_t)}{u_t-x^\star}.
\]
Multiplying the first inequality by $A_{t-1}$, the second by $a$, and adding them gives
\begin{align}
 A_tF(u_t)-A_{t-1}F(y_{t-1})
 &\le\ip{\nabla f(u_t)}{A_{t-1}(u_t-y_{t-1})+a(u_t-x^\star)}\nonumber\\
 &=a\ip{\nabla f(u_t)}{z_{t-1}-x^\star}.
 \label{eq:app-coupling}
\end{align}
The equality uses $A_tu_t=A_{t-1}y_{t-1}+az_{t-1}$, which is the definition of the query point. Combining \eqref{eq:app-local-progress}, \eqref{eq:app-quadratic-coefficient}, and \eqref{eq:app-coupling} yields
\begin{equation}
 A_tF(y_t)
 \le A_{t-1}F(y_{t-1})
 +a\ip{\nabla f(u_t)}{z_t-x^\star}
 +\frac1{16}\|z_t-z_{t-1}\|^2.
 \label{eq:app-objective-progress}
\end{equation}

\paragraph{Projection inequality.}
It remains to control the inner product in \eqref{eq:app-objective-progress}. Since
$z_t=\Pi_Q(z_{t-1}-a\widehat g_t)$ and $x^\star\in Q$, the optimality condition for Euclidean projection gives
\[
 \ip{z_t-z_{t-1}+a\widehat g_t}{x^\star-z_t}\ge0.
\]
Rearranging this inequality and applying the identity
$2\ip{z_{t-1}-z_t}{z_t-x^\star}
=\|z_{t-1}-x^\star\|^2-\|z_t-x^\star\|^2-\|z_t-z_{t-1}\|^2$
give
\begin{align}
 a\ip{\widehat g_t}{z_t-x^\star}
 \le{}&\frac12\|z_{t-1}-x^\star\|^2
 -\frac12\|z_t-x^\star\|^2-\frac12\|z_t-z_{t-1}\|^2.
 \label{eq:app-projection-progress}
\end{align}

\paragraph{Potential recursion.}
Substitute $\nabla f(u_t)=\widehat g_t-(\zeta_t+\beta_t)$ into \eqref{eq:app-objective-progress}, use \eqref{eq:app-projection-progress}, and add $\frac12\|z_t-x^\star\|^2$ to both sides. Splitting
$z_t-x^\star=(z_{t-1}-x^\star)+(z_t-z_{t-1})$ then gives
\begin{align*}
 \Phi_t
 \le{}&\Phi_{t-1}
 -a\ip{\zeta_t+\beta_t}{z_{t-1}-x^\star}
 -a\ip{\zeta_t+\beta_t}{z_t-z_{t-1}}-\left(\frac12-\frac1{16}\right)\|z_t-z_{t-1}\|^2.
\end{align*}
Young's inequality with coefficients $1$ and $1/4$ gives
\[
 -a\ip{\zeta_t+\beta_t}{z_t-z_{t-1}}
 \le a^2\|\zeta_t+\beta_t\|^2+\frac14\|z_t-z_{t-1}\|^2.
\]
After this substitution, the remaining coefficient of $\|z_t-z_{t-1}\|^2$ is
$1/16-1/2+1/4=-3/16\le0$. Discarding this nonpositive term proves
\[
 \Phi_t
 \le\Phi_{t-1}
 +a^2\|\zeta_t+\beta_t\|^2
 -a\ip{\zeta_t+\beta_t}{z_{t-1}-x^\star},
\]
which is exactly \eqref{eq:pathwise}.
\end{proof}

\section{Proof of Theorem~\ref{thm:phase}}\label{app:phaseproof}
\noindent
We now prove the one-phase contraction stated in Theorem~\ref{thm:phase}. The proof has four steps. First, the clipping threshold is chosen so that both the clipping bias and the conditional second moment of the centered clipped gradient have the required scale. Second, we stop the potential at the first exit from the localization region; the indicator of the stopped process is predictable, so the centered error terms remain martingale differences. Third, we apply a maximal Freedman inequality to the two martingale channels produced by the linear and quadratic error terms. Finally, the resulting uniform bound rules out an exit and converts the terminal potential bound into the desired objective-gap contraction.

Recall that
\[
 \Phi_t=A_tF(y_t)+\frac12\|z_t-x^\star\|^2,
 \qquad
 A_0=\frac{D^2}{\Delta},
 \qquad
 a=\frac{5D^2}{\Delta T}.
\]
Because $F(w)\le\Delta$, $y_0=z_0=w$, and $w,x^\star\in Q$ with $\diam(Q)=D$, the initial potential satisfies
\begin{equation}\label{eq:app-initial-potential}
 \Phi_0
 =A_0F(w)+\frac12\|w-x^\star\|^2
 \le D^2+\frac12D^2
 =\frac32D^2.
\end{equation}

\paragraph{Threshold choice.}
On the localized trajectory, Lemmas~\ref{lem:selfbound} and~\ref{lem:safe-queries} together with Assumption~\ref{ass:oracle} give
$\E_t\|g_t-\nabla f(u_t)\|^p\le\varsigma_{p,\Delta}^p$. Choose
\begin{equation}\label{eq:app-lambda}
 \Lambda_{\Delta,p}:=\max\left\{2G_\Delta,
 \left(\frac{64\varsigma_{p,\Delta}^pD}{\Delta}\right)^{1/(p-1)}\right\},
 \qquad R_{\Delta,p}:=\frac{D\Lambda_{\Delta,p}}{\Delta}.
\end{equation}
Equation~\ref{eq:app-lambda} combines two logically separate requirements: its first threshold localizes clipping around the true gradient, and its second threshold makes the clipping bias small relative to the target gap.
For the remainder of this proof, write $\Lambda:=\Lambda_{\Delta,p}$ and $R:=R_{\Delta,p}$. The first term in the maximum ensures $\Lambda\ge2G_\Delta$, which is the condition needed in Lemma~\ref{lem:clipping}. The second term is equivalent to
\begin{equation}\label{eq:app-threshold-power}
 \Lambda^{p-1}\ge\frac{64\varsigma_{p,\Delta}^pD}{\Delta}.
\end{equation}
Consequently, on every localized iteration, Lemma~\ref{lem:clipping} yields the predictable decomposition
\[
 \widehat g_t-\nabla f(u_t)=\zeta_t+\beta_t,
 \qquad \E_t\zeta_t=0.
\]
Indeed, substituting \eqref{eq:app-threshold-power} into the two moment bounds of Lemma~\ref{lem:clipping} gives
\begin{align}
 \|\beta_t\|
 &\le 2\varsigma_{p,\Delta}^p\Lambda^{1-p}
 \le\frac{\Delta}{32D},
 \label{eq:app-local-bias}\\
 v_t:=\E_t\|\zeta_t\|^2
 &\le2\varsigma_{p,\Delta}^p\Lambda^{2-p}
 =2\varsigma_{p,\Delta}^p\Lambda\Lambda^{1-p}
 \le\frac{\Delta\Lambda}{32D}.
 \label{eq:app-local-variance}
\end{align}

We also need an upper bound on the dimensionless threshold $R$. Since $\max\{x,y\}\le x+y$ for $x,y\ge0$, the definition of $\Lambda$ and the identity $G_\Delta=\sqrt{\Delta M}$, where $M=32C_\star\mathcal H_\Delta$, imply
\begin{align}
 R
 &=\frac{D\Lambda}{\Delta}\le \frac{2DG_\Delta}{\Delta}
 +\frac{D}{\Delta}
 \left(\frac{64\varsigma_{p,\Delta}^pD}{\Delta}\right)^{1/(p-1)}=2D\sqrt{\frac{M}{\Delta}}
 +64^{1/(p-1)}
 \left(\frac{\varsigma_{p,\Delta}D}{\Delta}\right)^q,
 \label{eq:app-R}
\end{align}
where $q=p/(p-1)$ and hence $1+1/(p-1)=q$.

\paragraph{Choice of the horizon.}
Set
\begin{equation}\label{eq:app-horizon}
 \ell_\rho:=\log\frac8\rho,
 \qquad
 T:=\left\lceil1024(1+R)\ell_\rho\right\rceil.
\end{equation}
The deterministic part of $R$ gives $R\ge2D\sqrt{M/\Delta}$. Since $\ell_\rho>1$, equation~\ref{eq:app-horizon} in particular implies $T\ge5D\sqrt{M/\Delta}$. Therefore
\begin{equation}\label{eq:app-phase-weight-check}
 Ma^2
 =\frac{25MD^4}{\Delta^2T^2}
 \le\frac{D^2}{\Delta}
 =A_0
 \le A_t.
\end{equation}
Condition~\eqref{eq:app-phase-weight-check} is precisely the weight condition used in Lemmas~\ref{lem:safe-queries} and~\ref{lem:one-step}.

\paragraph{Stopped martingale channels.}
Set the envelope level $B_{\rm env}:=3D^2$ and define the stopping time
\begin{equation}\label{eq:app-stopping-time}
 \tau:=\inf\{t\in\{1,\ldots,T\}:\Phi_t>B_{\rm env}\},
 \qquad \inf\varnothing:=T+1.
\end{equation}
The stopping rule in \eqref{eq:app-stopping-time} localizes every stochastic estimate to iterations for which the safe-query lemma is already valid.
For every $t$, let $I_t:=\mathbf1\{t-1<\tau\}$. The event $\{t-1<\tau\}$ is determined by $\Phi_0,\ldots,\Phi_{t-1}$, so $I_t$ is $\mathcal F_{t-1}$-measurable. Moreover, $I_t=1$ implies $\Phi_{t-1}\le B_{\rm env}$, and hence Lemmas~\ref{lem:safe-queries}--\ref{lem:one-step} and bounds \eqref{eq:app-local-bias}--\eqref{eq:app-local-variance} are valid at iteration $t$.

Multiplying \eqref{eq:pathwise} by $I_t$ and using
$\|\zeta_t+\beta_t\|^2\le2\|\zeta_t\|^2+2\|\beta_t\|^2$, we obtain
\begin{align}
 I_t(\Phi_t-\Phi_{t-1})
 \le{}&2a^2I_tv_t
 +2a^2I_t(\|\zeta_t\|^2-v_t)
 +2a^2I_t\|\beta_t\|^2\nonumber\\
 &-aI_t\ip{\zeta_t}{z_{t-1}-x^\star}
 -aI_t\ip{\beta_t}{z_{t-1}-x^\star}.
 \label{eq:app-stopped-rec}
\end{align}
The reason for adding and subtracting $v_t$ is that the second term on the right is now centered conditionally on $\mathcal F_{t-1}$. The fourth term is centered for the same reason, while the remaining three terms are predictable or deterministic upper bounds.

\paragraph{Predictable contributions.}
Because $z_{t-1},x^\star\in Q$, we have $\|z_{t-1}-x^\star\|\le D$. Using \eqref{eq:app-local-bias}, $Ta=5D^2/\Delta$, and $I_t\le1$ gives
\begin{align}
 \sum_{t=1}^T aI_tD\|\beta_t\|
 &\le TaD\frac{\Delta}{32D}
 =\frac5{32}D^2,
 \label{eq:app-det1}\\
 2\sum_{t=1}^Ta^2I_t\|\beta_t\|^2
 &\le2Ta^2\frac{\Delta^2}{32^2D^2}
 =\frac{50}{1024T}D^2.
 \label{eq:app-det2}
\end{align}
Similarly, \eqref{eq:app-local-variance} and $R=D\Lambda/\Delta$ yield
\begin{align}
 2\sum_{t=1}^Ta^2I_tv_t
 &\le2Ta^2\frac{\Delta\Lambda}{32D}
 =\frac{50}{32}\frac{R}{T}D^2.
 \label{eq:app-det3}
\end{align}
Equations~\ref{eq:app-horizon}, \ref{eq:app-det2}, and~\ref{eq:app-det3} imply
\begin{equation}\label{eq:app-small-predictable}
 \frac{50}{1024T}D^2
 +\frac{50}{32}\frac{R}{T}D^2
 <\frac1{100}D^2.
\end{equation}
Indeed, $T\ge1024\ell_\rho$ and $R/T\le[1024\ell_\rho]^{-1}$, while $\ell_\rho>1$.

\paragraph{Centered contributions.}
For $j\le T$, define the stopped scalar martingales
\[
 M_j^{(1)}:=-\sum_{t=1}^j aI_t
 \ip{\zeta_t}{z_{t-1}-x^\star},
 \qquad
 M_j^{(2)}:=2a^2\sum_{t=1}^jI_t(\|\zeta_t\|^2-v_t).
\]
Both are martingales because $I_t$ and $z_{t-1}$ are $\mathcal F_{t-1}$-measurable, $\E_t\zeta_t=0$, and $v_t=\E_t\|\zeta_t\|^2$. For the first channel, $\|\zeta_t\|\le2\Lambda$ gives
\begin{align}
 |\Delta M_t^{(1)}|
 &\le2a\Lambda D
 =\frac{10D^2R}{T},
 \label{eq:app-m1-increment}\\
 V_T^{(1)}
 &:=\sum_{t=1}^T\E_t[(\Delta M_t^{(1)})^2]
 \le a^2D^2\sum_{t=1}^TI_tv_t
 \le\frac{25}{32}\frac{D^4R}{T}.
 \label{eq:app-m1}
\end{align}
For the quadratic channel, both $\|\zeta_t\|^2$ and $v_t$ lie in $[0,4\Lambda^2]$, so
$|\|\zeta_t\|^2-v_t|\le4\Lambda^2$. In addition,
$\|\zeta_t\|^4\le4\Lambda^2\|\zeta_t\|^2$. Therefore
\begin{align}
 |\Delta M_t^{(2)}|
 &\le8a^2\Lambda^2
 =\frac{200D^2R^2}{T^2},
 \label{eq:app-m2-increment}\\
 V_T^{(2)}
 &:=\sum_{t=1}^T\E_t[(\Delta M_t^{(2)})^2]\le4a^4\sum_{t=1}^TI_t\E_t\|\zeta_t\|^4
 \le16a^4\Lambda^2\sum_{t=1}^TI_tv_t
 \le\frac{625}{2}\frac{D^4R^3}{T^3}.
 \label{eq:app-m2}
\end{align}

We use the following two-sided maximal form of Freedman's inequality: if a scalar martingale has increments bounded in absolute value by $L$ and predictable quadratic variation at most $V$, then, for every $u>0$,
\[
 \Prb\left\{\max_{j\le T}|M_j|\ge u\right\}
 \le2\exp\left(-\frac{u^2}{2(V+Lu/3)}\right).
\]
Take $u=D^2/8$. From \eqref{eq:app-horizon},
\[
 \frac{R}{T}\le\frac1{1024\ell_\rho},
 \qquad
 \frac{R^2}{T^2}\le\frac1{1024^2\ell_\rho^2},
 \qquad
 \frac{R^3}{T^3}\le\frac1{1024^3\ell_\rho^3}.
\]
Substitution of the linear increment and variation bounds
\eqref{eq:app-m1-increment} and~\eqref{eq:app-m1}, together with the quadratic bounds
\eqref{eq:app-m2-increment} and~\eqref{eq:app-m2}, shows, for each $i\in\{1,2\}$,
\[
 \frac{u^2}{2(V_T^{(i)}+L_i u/3)}\ge\ell_\rho.
\]
For example, for the linear channel this follows from
\[
 2\left(\frac{25}{32\cdot1024}
 +\frac{10}{24\cdot1024}\right)\le\frac1{64};
\]
the quadratic channel has still smaller normalized constants because it contains $1024^{-2}$ and $1024^{-3}$. Hence
\begin{equation}\label{eq:app-freedman-events}
 \Prb\left\{\max_{j\le T}|M_j^{(i)}|>\frac{D^2}{8}\right\}
 \le2e^{-\ell_\rho}=\frac\rho4,
 \qquad i=1,2.
\end{equation}

\paragraph{Closing the stopping argument.}
Let $\mathcal E$ be the event on which both martingales in \eqref{eq:app-freedman-events} remain within $D^2/8$. By a union bound, $\Prb(\mathcal E)\ge1-\rho/2\ge1-\rho$. For any $j\le T$, the definition of $I_t$ gives the telescoping identity
\[
 \sum_{t=1}^j I_t(\Phi_t-\Phi_{t-1})
 =\Phi_{j\wedge\tau}-\Phi_0.
\]
On $\mathcal E$, sum \eqref{eq:app-stopped-rec}, use \eqref{eq:app-initial-potential}, \eqref{eq:app-det1}, \eqref{eq:app-small-predictable}, and the two martingale bounds, and obtain, simultaneously for all $j\le T$,
\[
 \Phi_{j\wedge\tau}
 \le\frac32D^2+\frac5{32}D^2+\frac18D^2+\frac18D^2+\frac1{100}D^2<3D^2=B_{\rm env}.
\]
If $\tau\le T$, choosing $j=\tau$ would give $\Phi_\tau<B_{\rm env}$, contradicting the definition of $\tau$. Thus, on $\mathcal E$, necessarily $\tau=T+1$, and therefore $\Phi_T\le B_{\rm env}$.

Finally,
\[
 A_T=A_0+Ta
 =\frac{D^2}{\Delta}+\frac{5D^2}{\Delta}
 =\frac{6D^2}{\Delta}.
\]
Since the distance term in $\Phi_T$ is nonnegative,
\[
 F(y_T)\le\frac{\Phi_T}{A_T}\le\frac\Delta2.
\]
This conclusion holds on an event of probability at least $1-\rho$, which proves Theorem~\ref{thm:phase}.

\section{Proofs of the Main Convergence Results}\label{app:restart}
This section derives Theorems~\ref{thm:convex} and~\ref{thm:strong} from the one-phase result. We first obtain a common bound for the cost of a phase with gap level $\Delta$. We then sum that bound with a fixed set diameter for the convex result and with shrinking diameters for the strongly convex result. The only slightly nonstandard point is the $H_1$ crossover: after $H_1\Delta_s$ falls below $H_0$, the gap-dependent curvature must no longer be charged once per remaining phase.

\paragraph{Cost of one phase.}
By \eqref{eq:app-R}, $M=32C_\star\mathcal H_\Delta$, and
$T=\lceil1024(1+R_{\Delta,p})\log(8/\rho)\rceil$, the cost of one phase is bounded in terms of $\varsigma_{p,\Delta}$. To make this dependence explicit, recall that
\[
 \varsigma_{p,\Delta}^p
 =\sigma_0^p+\sigma_1^pG_\Delta^p
 +\sigma_2^p(4\Delta)^{p/2},
 \qquad
 G_\Delta^p
 =\bigl(32C_\star\Delta\mathcal H_\Delta\bigr)^{p/2}.
\]
Since $q/p=1/(p-1)\ge1$, the elementary inequality
$(u+v+w)^r\le3^{r-1}(u^r+v^r+w^r)$, valid for $r\ge1$ and nonnegative $u,v,w$, gives
\[
\left(\frac{\varsigma_{p,\Delta}D}{\Delta}\right)^q
\le3^{q/p-1}\left[
\left(\frac{\sigma_0D}{\Delta}\right)^q
+32^{q/2}\left(\sigma_1D\sqrt{C_\star\left(\frac{H_0}{\Delta}+H_1\right)}\right)^q
+2^q\left(\frac{\sigma_2D}{\sqrt\Delta}\right)^q\right].
\]
Consequently, for a constant $C_p>0$ depending only on $p$,
\begin{equation}\label{eq:app-phase-rate}
\begin{aligned}
T\le C_p\log\frac8\rho\Bigg[&1+D\sqrt{C_\star\!\left(\frac{H_0}{\Delta}+H_1\right)}
+\left(\frac{\sigma_0D}{\Delta}\right)^q\\
&+\left(\sigma_1D\sqrt{C_\star\!\left(\frac{H_0}{\Delta}+H_1\right)}\right)^q
+\left(\frac{\sigma_2D}{\sqrt\Delta}\right)^q\Bigg].
\end{aligned}
\end{equation}
For instance, $C_p=1024\,64^{q/p}3^{q/p-1}32^{q/2}$ is valid; these numerical factors are proof bounds, not tuning recommendations.

\paragraph{Restart events and geometric sums.}
Let $\Delta_s=2^{-s}\Delta_0$ and $S=\lceil\log_2(\Delta_0/\varepsilon)\rceil$. We repeatedly use
the following explicit sums. If $S=0$, then
$\varepsilon=\Delta_0$, the restart loop is empty, and
$w_S=w_0$ already satisfies $F(w_S)\le\varepsilon$. Thus all sums below are empty and both convergence statements are immediate. For the remainder of the restart analysis, assume $S\ge1$. The ceiling definition implies
\[
 2^{S-1}<\frac{\Delta_0}{\varepsilon}\le2^S.
\]
For every $a>0$, substituting
$\Delta_s=2^{-s}\Delta_0$ into the increasing geometric sum gives
\begin{align}
 \sum_{s=0}^{S-1}\Delta_s^{-a}
 &=\Delta_0^{-a}\sum_{s=0}^{S-1}2^{as}=\Delta_0^{-a}\frac{2^{aS}-1}{2^a-1}\le\frac{2^a}{2^a-1}\,\varepsilon^{-a}.
 \label{eq:app-geosum-increasing}
\end{align}
The last inequality uses
$2^S<2\Delta_0/\varepsilon$, which follows from
$2^{S-1}<\Delta_0/\varepsilon$. Likewise, the decreasing geometric sum is
\begin{align}
 \sum_{s=0}^{S-1}\Delta_s^a
 &=\Delta_0^a\sum_{s=0}^{S-1}2^{-as}=\Delta_0^a\frac{1-2^{-aS}}{1-2^{-a}}\le\frac{\Delta_0^a}{1-2^{-a}}.
 \label{eq:app-geosum-decreasing}
\end{align}
Finally, $\sum_{s=0}^{S-1}1=S$ because there are exactly $S$ phases.
A summable failure allocation, for example
$\rho_s=6\rho/[\pi^2(s+1)^2]$, gives
$\sum_{s\ge0}\rho_s=\rho$. To justify the restart induction explicitly, let
\[
 \mathcal E_s:=\{F(w_j)\le\Delta_j\text{ for }j=0,\ldots,s\}.
\]
The initial assumption gives $\mathcal E_0$. Conditional on $\mathcal E_s$, the set $Q_s$ contains $w_s$ and $x^\star$, so Theorem~\ref{thm:phase} gives
\[
 \Prb(\mathcal E_{s+1}^{\,c}\cap\mathcal E_s)\le\rho_s.
\]
A union bound over $s=0,\ldots,S-1$ therefore yields
\[
 \Prb(\mathcal E_S)
 \ge1-\sum_{s=0}^{S-1}\rho_s
 \ge1-\rho.
\]
On $\mathcal E_S$, $F(w_S)\le\Delta_S\le\varepsilon$. Moreover,
$\log(8/\rho_s)=O(\log(1/\rho)+\log(s+1))$; these confidence and phase-index logarithms are precisely the factors absorbed by $\widetilde O_{\rho,p}$.

\paragraph{A crossover lemma.}
The technical note used the valid but coarse inequality $\sqrt{a+b}\le\sqrt a+\sqrt b$, which yields an $H_1$ factor on every phase. The following split is strictly sharper whenever the algorithm crosses into the $H_0$-dominated regime.
We assign the equality case $H_1\Delta_s=H_0$ to the $H_0$-dominated regime and use $J_H$ as defined in Section~\ref{sec:convex-case}.

\begin{lemma}[$H_1$ crossover sums]\label{lem:crossover}
For every $r\ge1/2$,
\begin{equation}\label{eq:crossover-r}
 \sum_{s=0}^{S-1}\left(\frac{H_0}{\Delta_s}+H_1\right)^r
 \le C_r\left[\left(\frac{H_0}{\varepsilon}\right)^r+H_1^rJ_H\right].
\end{equation}
\end{lemma}

\begin{proof}
Fix $r\ge1/2$. If $H_1=0$, substituting $a=r$ into
\eqref{eq:app-geosum-increasing} proves the claim. If $H_0=0$, every summand equals $H_1^r$ and $J_H=S$, so both sides have the same dependence on $H_1^rS$. Assume now that $H_0>0$ and $H_1>0$, and split the phase indices into
$I_1=\{s:H_1>H_0/\Delta_s\}$ and
$I_0=\{s:H_1\le H_0/\Delta_s\}$. Since
$\Delta_s=2^{-s}\Delta_0$, the indices in $I_1$ satisfy
$s<\log_2(H_1\Delta_0/H_0)$, so their number is at most $J_H$. On $I_1$,
\[
 \sum_{s\in I_1}
 \left(\frac{H_0}{\Delta_s}+H_1\right)^r
 \le 2^rH_1^r|I_1|
 \le2^rH_1^rJ_H.
\]
On $I_0$, including the equality case,
$H_1\le H_0/\Delta_s$, and hence
\begin{align*}
 \sum_{s\in I_0}
 \left(\frac{H_0}{\Delta_s}+H_1\right)^r
 &\le2^rH_0^r\sum_{s\in I_0}\Delta_s^{-r}\le2^rH_0^r\sum_{s=0}^{S-1}\Delta_s^{-r}\le\frac{2^{2r}}{2^r-1}
 \left(\frac{H_0}{\varepsilon}\right)^r,
\end{align*}
where the last line applies \eqref{eq:app-geosum-increasing} with $a=r$. Adding the two regions proves \eqref{eq:crossover-r} with, for example,
$C_r=\max\{2^r,2^{2r}/(2^r-1)\}$.
\end{proof}

\paragraph{Proof of Theorem~\ref{thm:convex}.}
In the convex setting, the same set $Q$ and diameter $D$ are used in every phase. Summing the constant term in \eqref{eq:app-phase-rate} gives $S$. For the deterministic curvature term, Lemma~\ref{lem:crossover} with $r=1/2$ gives
\begin{align}
 \sum_{s=0}^{S-1}
 D\sqrt{C_\star\left(\frac{H_0}{\Delta_s}+H_1\right)}
 \le C D\left(
 \sqrt{\frac{H_0}{\varepsilon}}+\sqrt{H_1}\,J_H
 \right).
 \label{eq:app-cvx-deterministic-sum}
\end{align}
The additive and gap-dependent noise terms are direct geometric sums:
\begin{align}
 \sum_{s=0}^{S-1}\left(\frac{\sigma_0D}{\Delta_s}\right)^q
 &=(\sigma_0D)^q\Delta_0^{-q}
 \frac{2^{qS}-1}{2^q-1}
 \le\frac{2^q}{2^q-1}
 \left(\frac{\sigma_0D}{\varepsilon}\right)^q,
 \label{eq:app-cvx-additive-sum}\\
 \sum_{s=0}^{S-1}\left(\frac{\sigma_2D}{\sqrt{\Delta_s}}\right)^q
 &=(\sigma_2D)^q\Delta_0^{-q/2}
 \frac{2^{qS/2}-1}{2^{q/2}-1}\le\frac{2^{q/2}}{2^{q/2}-1}
 \left(\frac{\sigma_2D}{\sqrt{\varepsilon}}\right)^q.
 \label{eq:app-cvx-gap-sum}
\end{align}
Both last inequalities use
$2^S<2\Delta_0/\varepsilon$. Thus the constants depend on $p$ through $q$, but not on $\Delta_0$ or $\varepsilon$.
For the gradient-dependent term, apply Lemma~\ref{lem:crossover} with $r=q/2\ge1$:
\begin{align}
 \sum_{s=0}^{S-1}
 \left[
 \sigma_1D\sqrt{C_\star\left(\frac{H_0}{\Delta_s}+H_1\right)}
 \right]^q
 &\le
 C_p(\sigma_1D)^q
 \left[
 \left(\frac{H_0}{\varepsilon}\right)^{q/2}
 H_1^{q/2}J_H
 \right]\nonumber\\
 &\le
 C_p\left(
 \sigma_1D\left[
 \sqrt{\frac{H_0}{\varepsilon}}
 +\sqrt{H_1}\,J_H^{1/q}
 \right]\right)^q.
 \label{eq:app-cvx-multiplicative-sum}
\end{align}
The last inequality uses $x^q+y^q\le(x+y)^q$ for $x,y\ge0$ with
$x=\sqrt{H_0/\varepsilon}$ and
$y=\sqrt{H_1}\,J_H^{1/q}$. Combining the deterministic sum
\eqref{eq:app-cvx-deterministic-sum}, the additive sum
\eqref{eq:app-cvx-additive-sum}, the gap-dependent sum
\eqref{eq:app-cvx-gap-sum}, and the gradient-dependent sum
\eqref{eq:app-cvx-multiplicative-sum} with the restart event above gives exactly the oracle-complexity bound and the success probability stated in Theorem~\ref{thm:convex}. This sharper accounting changes only the restart summation; the one-phase theorem is unchanged.

\paragraph{Proof of Theorem~\ref{thm:strong}.}
On a successful phase, $F(w_s)\le\Delta_s$ and strong convexity imply
\[
 \frac{\mu}{2}\|w_s-x^\star\|^2
 \le F(w_s)\le\Delta_s,
\]
and hence $\|w_s-x^\star\|\le\sqrt{2\Delta_s/\mu}$. Set
\[
 Q_s=B\left(w_s,\sqrt{\frac{2\Delta_s}{\mu}}\right),
 \qquad D_s=2\sqrt{\frac{2\Delta_s}{\mu}}.
\]
Then $w_s,x^\star\in Q_s$, and $D_s$ is the diameter of $Q_s$, as required by Theorem~\ref{thm:phase}. Substitution into the deterministic phase term gives
\[
 D_s\sqrt{C_\star\left(\frac{H_0}{\Delta_s}+H_1\right)}
 =2\sqrt{\frac{2C_\star}{\mu}}
 \sqrt{H_0+H_1\Delta_s}
 \le2\sqrt{\frac{2C_\star}{\mu}}
 \left(\sqrt{H_0}+\sqrt{H_1\Delta_s}\right).
\]
The first part is constant across phases and therefore costs $S$. For the second part, substituting
$\Delta_s=2^{-s}\Delta_0$ gives
\[
 \sum_{s=0}^{S-1}\sqrt{\Delta_s}
 =\sqrt{\Delta_0}
 \frac{1-2^{-S/2}}{1-2^{-1/2}}
 \le\frac{\sqrt{\Delta_0}}{1-2^{-1/2}}.
\]
Thus the summed deterministic contribution is
\[
 O\left(
 \sqrt{\frac{H_0}{\mu}}\,S
 +\sqrt{\frac{H_1\Delta_0}{\mu}}
 \right),
\]
up to the absolute factor depending on $C_\star$.

For additive noise,
\[
 \left(\frac{\sigma_0D_s}{\Delta_s}\right)^q
 =\left(\frac{2\sqrt2\,\sigma_0}
 {\sqrt{\mu\Delta_s}}\right)^q,
\]
and substituting $\Delta_s=2^{-s}\Delta_0$ gives
\[
 \sum_{s=0}^{S-1}
 \left(\frac{\sigma_0D_s}{\Delta_s}\right)^q
 =\left(\frac{2\sqrt2\,\sigma_0}{\sqrt\mu}\right)^q
 \Delta_0^{-q/2}
 \frac{2^{qS/2}-1}{2^{q/2}-1}
 \le C_p
 \left(\frac{\sigma_0}{\sqrt{\mu\varepsilon}}\right)^q.
\]
Here the last inequality again uses
$2^S<2\Delta_0/\varepsilon$.
For gap-dependent noise,
\[
 \left(\frac{\sigma_2D_s}{\sqrt{\Delta_s}}\right)^q
 =\left(\frac{2\sqrt2\,\sigma_2}{\sqrt\mu}\right)^q
\]
is constant across phases. Summing exactly $S$ identical terms therefore contributes
\[
 \left(\frac{2\sqrt2\,\sigma_2}{\sqrt\mu}\right)^qS.
\]

For the gradient-dependent contribution, substitution of $D_s$ gives
\[
 \left[\sigma_1D_s\sqrt{C_\star\left(\frac{H_0}{\Delta_s}+H_1\right)}\right]^q
 =\left(\frac{2\sqrt{2C_\star}\,\sigma_1}{\sqrt\mu}\right)^q
 (H_0+H_1\Delta_s)^{q/2}.
\]
Since $q/2\ge1$, $(x+y)^{q/2}\le2^{q/2-1}(x^{q/2}+y^{q/2})$. Therefore
\[
 \sum_{s=0}^{S-1}\Delta_s^{q/2}
 =\Delta_0^{q/2}
 \frac{1-2^{-qS/2}}{1-2^{-q/2}}
 \le\frac{\Delta_0^{q/2}}{1-2^{-q/2}},
\]
while $\sum_{s=0}^{S-1}H_0^{q/2}=H_0^{q/2}S$. Using these two identities,
\begin{align}
 \sum_{s=0}^{S-1}
 \left[\sigma_1D_s\sqrt{C_\star\left(\frac{H_0}{\Delta_s}+H_1\right)}\right]^q
 &\le
 C_p\left(\frac{\sigma_1}{\sqrt\mu}\right)^q
 \left(H_0^{q/2}S+(H_1\Delta_0)^{q/2}\right)\nonumber\\
 &\le
 C_p\left(
 \frac{\sigma_1}{\sqrt\mu}
 \left[\sqrt{H_0}\,S^{1/q}
 +\sqrt{H_1\Delta_0}\right]\right)^q.
 \label{eq:app-sc-multiplicative-sum}
\end{align}
The last line again uses $x^q+y^q\le(x+y)^q$. Equation~\ref{eq:app-sc-multiplicative-sum} is the gradient-dependent contribution in Theorem~\ref{thm:strong}. Adding the constant phase cost $S$, the deterministic contribution, and the three stochastic sums proves the complexity stated there. The restart-event argument gives success probability at least $1-\rho$. Finally, if $\sigma_0=0$, every remaining phase contribution is either constant in $s$ or summable from the initial phase, which explains the logarithmic target-accuracy dependence emphasized in the main text.

\section{Power-Law Noise and Finite-Sum Derivation}\label{app:structural}
\subsection{Power-law phase diagram}
\paragraph{Localized phase scale.}
We prove Theorem~\ref{thm:alpha} by isolating the only part of the one-phase cost that changes under \eqref{eq:Falpha}. On the localized trajectory, Lemma~\ref{lem:safe-queries} gives $F(u_t)\le4\Delta$. Hence
\[
 \E_t\|g(u_t,\xi_t)-\nabla f(u_t)\|^p
 \le\tau_\alpha^p(4\Delta)^\alpha,
\]
so the phase moment scale may be chosen as
\[
 \varsigma_{p,\Delta}
 =4^{\alpha/p}\tau_\alpha\Delta^{\alpha/p}.
\]
Substitution into the stochastic term of \eqref{eq:app-R} gives
\[
 \left(\frac{\varsigma_{p,\Delta}D}{\Delta}\right)^q
 =4^{\alpha q/p}(\tau_\alpha D)^q
 \Delta^{-q+\alpha q/p}.
\]
Because $q=p/(p-1)$ and $q/p=1/(p-1)$,
\[
 -q+\frac{\alpha q}{p}
 =\frac{\alpha-p}{p-1}.
\]
Therefore, in the convex case, the stochastic cost of phase $s$ is, up to constants depending only on $p$ and $\alpha$,
\begin{equation}\label{eq:alpha-phase-cvx}
 (\tau_\alpha D)^q
 \Delta_s^{(\alpha-p)/(p-1)}.
\end{equation}
\paragraph{Convex restart.}
There are three cases:
\begin{itemize}
\item If $\alpha<p$, set
$a=(p-\alpha)/(p-1)>0$. Substituting
$\Delta_s=2^{-s}\Delta_0$ gives
\[
 \sum_{s=0}^{S-1}\Delta_s^{-a}
 =\Delta_0^{-a}\frac{2^{aS}-1}{2^a-1}
 \le\frac{2^a}{2^a-1}\varepsilon^{-a}.
\]
Thus the last phases determine the order
$\varepsilon^{-(p-\alpha)/(p-1)}$.
\item If $\alpha=p$, the exponent is zero. Every phase has the same stochastic cost, so the sum contributes $S=O(\log(\Delta_0/\varepsilon))$.
\item If $\alpha>p$, set
$a=(\alpha-p)/(p-1)>0$. Then
\[
 \sum_{s=0}^{S-1}\Delta_s^a
 =\Delta_0^a\frac{1-2^{-aS}}{1-2^{-a}}
 \le\frac{\Delta_0^a}{1-2^{-a}}.
\]
Thus the initial phases determine the order
$\Delta_0^{(\alpha-p)/(p-1)}$.
\end{itemize}
These are exactly the three branches of $\mathcal K_p(\alpha)$ in Theorem~\ref{thm:alpha}.

\paragraph{Strongly convex restart.}
Under strong convexity, the phase diameter is
$D_s=2\sqrt{2\Delta_s/\mu}$. Consequently,
\[
 D_s^q
 =2^{3q/2}\mu^{-q/2}\Delta_s^{q/2},
\]
and multiplying \eqref{eq:alpha-phase-cvx} by the additional factor
$\Delta_s^{q/2}/(\mu^{q/2}D^q)$ changes the exponent to
\[
 \frac{\alpha-p}{p-1}+\frac q2
 =\frac{\alpha-p/2}{p-1}.
\]
Thus the strongly convex phase cost is
\begin{equation}\label{eq:alpha-phase}
 C_{p,\alpha}
 \left(\frac{\tau_\alpha}{\sqrt\mu}\right)^q
 \Delta_s^{(\alpha-p/2)/(p-1)}.
\end{equation}
For $\alpha<p/2$, substituting
$a=(p/2-\alpha)/(p-1)$ into
\eqref{eq:app-geosum-increasing} gives
\[
 \sum_{s=0}^{S-1}
 \Delta_s^{(\alpha-p/2)/(p-1)}
 \le\frac{2^a}{2^a-1}\varepsilon^{-a}.
\]
For $\alpha=p/2$, the exponent is zero and the sum equals $S$. For
$\alpha>p/2$, substituting
$a=(\alpha-p/2)/(p-1)$ into
\eqref{eq:app-geosum-decreasing} gives
\[
 \sum_{s=0}^{S-1}
 \Delta_s^{(\alpha-p/2)/(p-1)}
 \le\frac{\Delta_0^a}{1-2^{-a}}.
\]
These are the three branches of $\mathcal K_{p/2}(\alpha)$.

\paragraph{Total complexity.}
The deterministic part of each phase is unchanged, and the restart-event argument from Appendix~\ref{app:restart} supplies the probability $1-\rho$. Returning to the claim of Theorem~\ref{thm:alpha}, the convex and strongly convex stochastic contributions are therefore exactly the two bounds stated there, which completes the proof.

\subsection{Finite-sum interpolation}
\paragraph{Component-growth assumptions.}
Let $f(x)=\E_\xi f_\xi(x)$, $g(x,\xi)=\nabla f_\xi(x)$, and fix $x^\star\in\arg\min f$. Suppose
\begin{equation}\label{eq:component-growth}
 \E\|g(x^\star,\xi)\|^p\le a_0^p,
 \qquad
 \E\|g(x,\xi)-g(x^\star,\xi)\|^p
 \le a_1^p\|\nabla f(x)\|^p+a_2^pF(x)^{p/2}.
\end{equation}
Then Assumption~\ref{ass:oracle} holds with
\begin{equation}\label{eq:finite-constants}
 \sigma_0^p=2^{p-1}a_0^p,\qquad
 \sigma_1^p=2^{2p-1}a_1^p,\qquad
 \sigma_2^p=2^{2p-1}a_2^p.
\end{equation}
\paragraph{Centered decomposition.}
We verify the reduction directly. Let
$A(x,\xi)=g(x,\xi)-g(x^\star,\xi)$ and $B(\xi)=g(x^\star,\xi)$. Since $\E B=\nabla f(x^\star)=0$ and $\E A=\nabla f(x)$,
\[
 g(x,\xi)-\nabla f(x)=(A-\E A)+B.
\]
\paragraph{Conditional moment bound.}
The equality $\E B=0$ uses differentiability of $f$ and first-order optimality at the unconstrained minimizer $x^\star$. For $p\in(1,2]$,
$\|u+v\|^p\le2^{p-1}(\|u\|^p+\|v\|^p)$, so
\[
 \E\|g(x,\xi)-\nabla f(x)\|^p\le2^{p-1}\E\|A-\E A\|^p+2^{p-1}\E\|B\|^p.
\]
Again by the same inequality and Jensen's inequality,
\[
 \E\|A-\E A\|^p
 \le2^{p-1}(\E\|A\|^p+\|\E A\|^p)
 \le2^p\E\|A\|^p.
\]
The second inequality follows because
$\|\E A\|^p\le\E\|A\|^p$. Substituting
\eqref{eq:component-growth} yields
\[
 \E\|g(x,\xi)-\nabla f(x)\|^p
 \le2^{p-1}a_0^p+2^{2p-1}a_1^p\|\nabla f(x)\|^p
 +2^{2p-1}a_2^pF(x)^{p/2},
\]
which proves \eqref{eq:finite-constants}. For a finite-sum oracle sampled independently at every iteration, the same calculation holds conditionally on $\mathcal F_{t-1}$ because the query point is then fixed and the new index is fresh. Hence the resulting constants satisfy the conditional formulation of Assumption~\ref{ass:oracle}.

\paragraph{Interpolation endpoint.}
In particular, exact interpolation gives $g(x^\star,\xi)=0$ almost surely, so $a_0=\sigma_0=0$ under \eqref{eq:component-growth}. Theorem~\ref{thm:strong} then has logarithmic target-accuracy dependence, even when $p<2$.

If
$\E\|g(x,\xi)-g(x^\star,\xi)\|^p\le L_p^{p/2}F(x)^{p/2}$ and
$\E\|g(x^\star,\xi)\|^p\le\sigma_{\star,p}^p$, then
$\sigma_1=0$, $\sigma_0^p\le2^{p-1}\sigma_{\star,p}^p$, and
$\sigma_2^p\le2^{2p-1}L_p^{p/2}$. At $p=2$ this is the familiar component-difference expected-smoothness scale up to convention-dependent constants \citep{gower2019sgd}; for $p<2$ it should be read as a finite-$p$ moment analogue rather than as a unique standard definition.

\section{Lower Bounds on the Smooth Subclass}\label{app:lower}
Throughout this section $H_1=0$, hence Assumption~\ref{ass:h01} is the standard global smoothness upper model with $L=2H_0$. Let $q=p/(p-1)$. The purpose of restricting to this subclass is logical: a lower bound on a subclass is automatically a lower bound for the full $(H_0,H_1)$ class. We treat deterministic curvature, additive finite-$p$ noise, and gap-dependent noise separately so that each term in Theorems~\ref{thm:convex} and~\ref{thm:strong} can be compared with a matching hard instance.

\subsection{Deterministic curvature}\label{app:lower-deterministic}
The standard first-order lower bounds \citep{nesterov2004} imply, in sufficiently high dimension,
\[
 N=\Omega\left(D\sqrt{\frac{H_0}{\varepsilon}}\right)
\]
for convex objectives and
\[
 N=\Omega\left(\sqrt{\frac{H_0}{\mu}}\log\frac{\Delta_0}{\varepsilon}\right)
\]
for $\mu$-strongly convex objectives in the nontrivial condition-number regime. These match the corresponding deterministic parts of Theorems~\ref{thm:convex}--\ref{thm:strong} when $H_1=0$.

\subsection{Additive finite-$p$ noise}\label{app:lower-additive}
We use a two-point testing construction in which information about the sign of the minimizer appears only through a rare oracle response. The rarity of that response is chosen to saturate the $p$th-moment constraint.

For $v\in\{-1,+1\}$, $\delta>0$, and $\eta\in(0,1/2]$, define
\begin{equation}\label{eq:rare-signal}
 Z_v=\begin{cases}-v\delta/\eta,&\text{with probability }\eta,\\0,&\text{with probability }1-\eta.\end{cases}
\end{equation}
Then $\E Z_v=-v\delta$. On the rare event, the centered value
$Z_v-\E Z_v=Z_v+v\delta$ has magnitude
$\delta(1-\eta)/\eta$; on the complementary event, it has magnitude $\delta$. Therefore
\begin{equation}\label{eq:rare-moment}
 \E|Z_v+v\delta|^p
 =\delta^p\bigl[(1-\eta)+\eta^{1-p}(1-\eta)^p\bigr]
 \le2\delta^p\eta^{1-p}.
\end{equation}
Choosing
\begin{equation}\label{eq:eta-add}
 \eta=2^{1/(p-1)}\left(\frac{\delta}{\sigma_0}\right)^q
\end{equation}
ensures $\E|Z_v+v\delta|^p\le\sigma_0^p$ whenever $\eta\le1/2$. Indeed,
\[
 2\delta^p\eta^{1-p}
 =2\delta^p
 \left[
 2^{1/(p-1)}
 \left(\frac{\delta}{\sigma_0}\right)^q
 \right]^{1-p}
 =\sigma_0^p.
\]

\begin{theorem}[Convex additive lower bound]\label{thm:lb-add-cvx}
Let $\rho\le1/4$, $\varepsilon\le H_0D^2/32$, and suppose the $\eta$ in \eqref{eq:eta-add} with $\delta=16\varepsilon/D$ is at most $1/2$. Any stochastic first-order method that succeeds uniformly with probability $1-\rho$ on the subclass $H_1=\sigma_1=\sigma_2=0$ requires
\[
 T\ge c_p\left(\frac{\sigma_0D}{\varepsilon}\right)^q\log\frac1{2\rho}.
\]
\end{theorem}

\begin{proof}
We construct two one-dimensional quadratics indexed by a hidden sign $v$. We verify their membership in the smooth subclass and the oracle moment condition, show that an accurate point reveals $v$, and then lower-bound the number of calls needed to observe the rare sign-dependent response.

\paragraph{Hard instances.}
Take $Q=[-D/2,D/2]$, $x_0=0$, $r=D/4$, and
\[
 f_v(x)=\frac\lambda2(x-vr)^2,
 \qquad \lambda=\frac{64\varepsilon}{D^2},
 \qquad \delta=\lambda r=\frac{16\varepsilon}{D}.
\]
The minimizer is $x_v^\star=vr\in Q$, and
\[
 F_v(x)=f_v(x)-f_v(x_v^\star)
 =\frac{\lambda}{2}(x-vr)^2.
\]
The Hessian is the scalar $\lambda$. Since
$\varepsilon\le H_0D^2/32$,
\[
 \frac{\lambda}{2}
 =\frac{32\varepsilon}{D^2}
 \le H_0,
\]
so $f_v$ satisfies Assumption~\ref{ass:h01} with $H_1=0$.

\paragraph{Oracle validity.}
Use the oracle
\[
 g_v(x):=\lambda x+Z_v.
\]
Because $\delta=\lambda r$,
\[
 \E g_v(x)
 =\lambda x-v\delta
 =\lambda(x-vr)
 =\nabla f_v(x),
\]
and its centered error is $Z_v-\E Z_v=Z_v+v\delta$. Thus \eqref{eq:rare-moment}--\eqref{eq:eta-add} verify Assumption~\ref{ass:oracle} with $\sigma_1=\sigma_2=0$.

\paragraph{Reduction to sign testing.}
If $vx\le0$, then $x$ lies on the wrong side of the origin and
$|x-vr|\ge r$. Consequently,
\[
 F_v(x)\ge\frac{\lambda r^2}{2}=2\varepsilon.
\]
In particular, any output satisfying $F_v(x)\le\varepsilon$ must have the same sign as $v$ and therefore identifies the hidden instance.

Now place the uniform prior on $v\in\{-1,+1\}$. Let $\mathcal N_T$ be the event that none of the first $T$ oracle calls produces the rare value in \eqref{eq:rare-signal}. Since the rare-event indicators are independent of the adaptive query points,
\[
 \Prb(\mathcal N_T)=(1-\eta)^T.
\]
On $\mathcal N_T$, every response equals $\lambda x$, regardless of $v$. Thus the entire adaptive transcript, and hence the algorithm's output distribution, is identical for the two signs. Conditional on $\mathcal N_T$, the probability of identifying $v$ is at most $1/2$. The average failure probability is therefore at least
$\frac12(1-\eta)^T$, so the worst-case failure probability over the two instances is at least the same quantity.

\paragraph{Number of oracle calls.}
If the method has failure probability at most $\rho$ on both instances, then
$(1-\eta)^T\le2\rho$. Taking logarithms gives
\[
 T\ge
 \frac{\log(1/(2\rho))}{-\log(1-\eta)}
 \ge\frac{1}{2\eta}\log\frac1{2\rho},
\]
where the last inequality uses
$-\log(1-\eta)\le2\eta$ for $\eta\le1/2$.
Finally, substituting \eqref{eq:eta-add} with
$\delta=16\varepsilon/D$ proves the stated lower bound, with a constant depending only on $p$.
\end{proof}

\begin{theorem}[Strongly convex additive lower bound]\label{thm:lb-add-sc}
Let $\rho\le1/4$ and $H_0\ge\mu/2$. Define $\delta=2\sqrt{\mu\varepsilon}$ and let $\eta$ be given by \eqref{eq:eta-add}. If $\eta\le1/2$, any method with uniform success probability $1-\rho$ on the subclass $H_1=\sigma_1=\sigma_2=0$ requires
\[
 T\ge c_p\left(\frac{\sigma_0}{\sqrt{\mu\varepsilon}}\right)^q\log\frac1{2\rho}.
\]
\end{theorem}
\begin{proof}
We reuse the rare-signal oracle with quadratic curvature $\mu$. The proof verifies strong convexity and oracle validity, reduces optimization to sign testing, and substitutes the strong-convexity scale into the testing bound.

\paragraph{Hard instances and oracle validity.}
Use
\[
 f_v(x)=\frac\mu2(x-vr)^2,
 \qquad
 r=2\sqrt{\frac{\varepsilon}{\mu}},
\]
so $\delta=\mu r=2\sqrt{\mu\varepsilon}$. The functions are $\mu$-strongly convex, and their quadratic upper model has coefficient $\mu/2\le H_0$. Thus they belong to the asserted subclass. If $vx\le0$, then
$F_v(x)\ge\mu r^2/2=2\varepsilon$, so an $\varepsilon$-accurate output again identifies $v$. With the oracle
$g_v(x)=\mu x+Z_v$, the unbiasedness and moment calculation use
\eqref{eq:rare-moment}--\eqref{eq:eta-add}, with $\lambda$ replaced by $\mu$.

\paragraph{Testing reduction and final complexity.}
On the event with no rare response, the transcript is independent of $v$, exactly as in the convex construction. Therefore the sign-testing calculation gives
$T\ge[2\eta]^{-1}\log(1/(2\rho))$. Substitution of
$\delta=2\sqrt{\mu\varepsilon}$ into \eqref{eq:eta-add} then gives
$\eta^{-1}=\Theta_p((\sigma_0/\sqrt{\mu\varepsilon})^q)$ and proves the claim.
\end{proof}

\subsection{Gap/interpolation noise}\label{app:lower-gap}
Set $\sigma_0=\sigma_1=0$ and use Bernoulli thinning of the exact gradient:
\begin{equation}\label{eq:thinning}
 g(x)=\frac{B}{\eta}\nabla f(x),\qquad B\sim\Ber(\eta).
\end{equation}
This oracle is unbiased because $\E B=\eta$. Its centered error equals
$(B/\eta-1)\nabla f(x)$. If $\eta\le1/2$, direct evaluation of the two Bernoulli outcomes gives
\begin{equation}\label{eq:thin-moment}
 \begin{aligned}
 \E_t\|g(x)-\nabla f(x)\|^p
 &=\left[(1-\eta)+\eta^{1-p}(1-\eta)^p\right]
 \|\nabla f(x)\|^p\\
 &\le2\eta^{1-p}\|\nabla f(x)\|^p.
 \end{aligned}
\end{equation}

\begin{theorem}[Convex gap-noise lower bound]\label{thm:lb-gap-cvx}
Let $\rho\le1/4$, $\varepsilon\le H_0D^2/32$, and set
$\lambda=64\varepsilon/D^2$. Choose $\eta$ so that
\[
 2\eta^{1-p}(2\lambda)^{p/2}=\sigma_2^p,
\]
and assume the resulting $\eta$ is at most $1/2$. Any method with uniform success probability $1-\rho$ on the subclass $H_1=\sigma_0=\sigma_1=0$ requires
\[
 T\ge c_p\left(\frac{\sigma_2D}{\sqrt\varepsilon}\right)^q\log\frac1{2\rho}.
\]
\end{theorem}
\begin{proof}
We use the convex quadratic pair from Theorem~\ref{thm:lb-add-cvx}, replace the additive rare signal by Bernoulli thinning, verify the gap-dependent conditional moment bound, and then repeat the sign-testing reduction with its probability written explicitly.

\paragraph{Conditional moment bound.}
Use the same two quadratics as in Theorem~\ref{thm:lb-add-cvx}. For a quadratic,
$|\nabla f_v(x)|^2=2\lambda F_v(x)$, and hence
\[
 |\nabla f_v(x)|^p=(2\lambda F_v(x))^{p/2}.
\]
Therefore \eqref{eq:thin-moment} and the definition of $\eta$ give
\[
 \E_t|g(x)-\nabla f_v(x)|^p
 \le\sigma_2^pF_v(x)^{p/2}.
\]
Solving for $\eta$ yields
\[
 \eta^{-1}
 =\Theta_p\left(\frac{\sigma_2}{\sqrt\lambda}\right)^q
 =\Theta_p\left(\frac{\sigma_2D}{\sqrt\varepsilon}\right)^q.
\]
\paragraph{Testing reduction.}
If $B_t=0$ at every call, every oracle response is zero, independently of the sign $v$. The full adaptive transcript then contains no information about the minimizer. Since an $\varepsilon$-accurate output must identify $v$, the same calculation as in Theorem~\ref{thm:lb-add-cvx} gives
the following explicit testing inequality:
\[
 \rho
 \ge \frac12\Prb\{B_1=\cdots=B_T=0\}
 =\frac12(1-\eta)^T.
\]
Taking logarithms gives
$T\log(1/(1-\eta))\ge\log(1/(2\rho))$.
Because $\eta\le1/2$ implies
$\log(1/(1-\eta))\le\eta/(1-\eta)\le2\eta$, we conclude that
$T\ge[2\eta]^{-1}\log(1/(2\rho))$. Combining this inequality with the displayed value of $\eta^{-1}$ proves the claim.
\end{proof}

\begin{lemma}[Thinning transfers exact-gradient lower bounds]\label{lem:thin-transfer}
Suppose a hard distribution over functions has the property that every exact first-order method using fewer than $K$ informative gradients has average failure at least $c_0>0$. Replacing the exact oracle by \eqref{eq:thinning}, any stochastic method with constant success probability bounded away from $1-c_0$ needs $T=\Omega(K/\eta)$ calls.
\end{lemma}
\begin{proof}
Let $S_T$ denote the number of nonzero thinned responses among $T$ calls. We first couple the stochastic transcript to an exact-gradient transcript containing only $S_T$ informative queries, then control $S_T$ and invoke the assumed exact-gradient lower bound.

\paragraph{Transcript simulation.}
The number of nonzero answers is
$S_T:=\sum_{t=1}^TB_t\sim\mathrm{Binomial}(T,\eta)$, even though the query points may be adaptive, because the Bernoulli variables are fresh and independent. Conditional on the Bernoulli sequence, every nonzero response is
$\nabla f(x_t)/\eta$ and therefore reveals exactly the same information as an exact gradient, while a zero response contains no instance-dependent information. Thus the stochastic transcript can be simulated by an exact-oracle method making only $S_T$ informative queries.

\paragraph{Informative-query count.}
If $T<K/(2\eta)$, then $\E S_T=T\eta<K/2$, and Markov's inequality gives
$\Prb\{S_T\ge K\}\le1/2$. On the event $\{S_T<K\}$, the assumed exact-oracle lower bound gives conditional average failure at least $c_0$. Hence the unconditional average failure is at least $c_0/2$. This proves that a constant success probability bounded away from $1-c_0/2$ requires $T=\Omega(K/\eta)$. Yao's minimax principle converts the distributional statement into a randomized worst-case lower bound.
\end{proof}

\begin{theorem}[Strongly convex gap-noise lower bound]\label{thm:lb-gap-sc}
There is a universal $c>0$ such that, on a fixed-condition-number smooth strongly convex subclass (for example, $L=4\mu$, corresponding to $H_0=2\mu$), and in the regime $\sigma_2\gtrsim_p\sqrt\mu$, any stochastic first-order method with constant success probability requires
\[
 T\ge c_p\left(\frac{\sigma_2}{\sqrt\mu}\right)^q\log\frac{\Delta_0}{\varepsilon}.
\]
For arbitrary confidence, an additional term of order $\eta^{-1}\log(1/\rho)$ is unavoidable; no claim is made that all upper and lower confidence logarithms coincide.
\end{theorem}
\begin{proof}
We start from a fixed-condition-number deterministic zero-chain requiring $K$ informative gradients. We then verify that Bernoulli thinning satisfies the gap-dependent moment model, apply Lemma~\ref{lem:thin-transfer}, and finally account for the confidence-dependent probability of receiving no informative response.

\paragraph{Deterministic hard family.}
Take an $L$-smooth, $\mu$-strongly convex zero-chain family with $L=4\mu$ from the first-order lower-bound construction of \citet{nesterov2004}. For this fixed condition number, the construction states that, in sufficiently high dimension, every exact first-order method needs
\[
 K=\Omega\left(\log\frac{\Delta_0}{\varepsilon}\right)
\]
informative gradients to reduce the initial gap $\Delta_0$ to $\varepsilon$. Because our convention is $L=2H_0$, choosing $H_0=2\mu$ places the family in the subclass of Assumption~\ref{ass:h01}.

\paragraph{Conditional moment bound for the thinned oracle.}
For every $L$-smooth convex function, apply its global upper model at
$x-\nabla f(x)/L$. Comparing the resulting value with $f^\star$ gives
$\|\nabla f(x)\|^2\le2LF(x)$. Substituting this inequality into \eqref{eq:thin-moment} yields
\[
 \E_t\|g(x_t)-\nabla f(x_t)\|^p\le2\eta^{1-p}(2L)^{p/2}F(x_t)^{p/2}.
\]
Choose $\eta$ by
\[
 2\eta^{1-p}(2L)^{p/2}=\sigma_2^p.
\]
The assumption $\sigma_2\gtrsim_p\sqrt\mu$ ensures that the resulting $\eta$ lies in $(0,1/2]$. Since $L=4\mu$,
\[
 \eta^{-1}
 =\Theta_p\left(\frac{\sigma_2}{\sqrt L}\right)^q
 =\Theta_p\left(\frac{\sigma_2}{\sqrt\mu}\right)^q.
\]
\paragraph{Transfer and confidence.}
Lemma~\ref{lem:thin-transfer} now gives
$T=\Omega(K/\eta)$ and proves the constant-confidence term. For confidence $1-\rho$, even a two-instance test cannot succeed before at least one informative Bernoulli outcome appears; the probability of no such outcome is $(1-\eta)^T$. Repeating the logarithmic calculation from Theorem~\ref{thm:lb-add-cvx} yields the additional necessary term
$\Omega(\eta^{-1}\log(1/\rho))$.
\end{proof}

\section{Multiplicative Noise: Universality versus Specialized Acceleration}\label{app:sigma1}
Set $H_1=\sigma_0=\sigma_2=0$ and suppose
\begin{equation}\label{eq:app-mult}
 \E_t\|g(x_t)-\nabla f(x_t)\|^p\le\sigma_1^p\|\nabla f(x_t)\|^p.
\end{equation}
The thinned oracle \eqref{eq:thinning} satisfies this condition when $2\eta^{1-p}\le\sigma_1^p$, so the transfer lemma gives the information-theoretic benchmark
\begin{equation}\label{eq:app-mult-lb}
 T=\Omega_p\bigl((1+\sigma_1^q)K\bigr)
\end{equation}
Equation~\ref{eq:app-mult-lb} is the information-theoretic benchmark against which we compare the universal phase construction below.
Here $K$ denotes any deterministic exact-gradient lower complexity at constant confidence.

The universal phase proof instead bounds $\|\nabla f\|$ by a phase constant before robustification; its sample cost can then contain $K+(\sigma_1K)^q$. Averaging $b$ stochastic gradients does not remove this term within the same argument. A finite-$p$ batch mean decreases the relative scale as $b^{-1/q}$, so the number of accelerated outer iterations becomes
$K+\sigma_1^qK^q/b$, while raw samples equal
\begin{equation}\label{eq:app-batching}
 bK+\sigma_1^qK^q.
\end{equation}
Equation~\ref{eq:app-batching} shows that the problematic term is independent of $b$.

At $p=2$, specialized AGNES dynamics achieves $O((1+\sigma_1^2)K)$ in expectation for smooth convex and strongly convex minimization \citep{gupta2024agnes}, showing that the benchmark is algorithmically meaningful. For $1<p<2$, robust mean estimation gives the following partial reduction.

\begin{lemma}[Robust relative gradient estimator]\label{lem:relative-macro}
At a predictable point $x$, take $m$ fresh samples satisfying \eqref{eq:app-mult}. There is a geometric median-of-means estimator $\widehat g(x)$ such that with probability at least $1-\delta$,
\[
 \|\widehat g(x)-\nabla f(x)\|
 \le K_p\sigma_1\left(\frac{\log(2/\delta)}{m}\right)^{1/q}\|\nabla f(x)\|.
\]
Thus constant relative accuracy uses $m=O_p(\sigma_1^q\log(1/\delta))$ samples.
\end{lemma}
\begin{proof}
Conditionally on the history, we identify the finite-$p$ moment scale of the fresh stochastic gradients and apply the geometric median-of-means lemma. We then rewrite its exponent as $1/q$ and solve the resulting inequality for a prescribed relative accuracy.

\paragraph{Conditional reduction.}
Condition on the history that determines the predictable point $x$. The fresh stochastic gradients are then independent random vectors with conditional mean $\nabla f(x)$ and conditional $p$th central moment at most
\[
 \nu^p:=\sigma_1^p\|\nabla f(x)\|^p.
\]
Apply Lemma~\ref{lem:GMoM} below with this value of $\nu$ and $n=m$. It gives
\[
 \|\widehat g(x)-\nabla f(x)\|
 \le K_p\nu
 \left(\frac{\log(2/\delta)}{m}\right)^{(p-1)/p}.
\]
Since $(p-1)/p=1/q$ and $\nu=\sigma_1\|\nabla f(x)\|$, this is the claimed inequality. To make the relative error at most any fixed constant $c\in(0,1)$, it is enough to take
$m\ge C_{p,c}\sigma_1^q\log(2/\delta)$.
\end{proof}
This shows that the heavy-tailed mean-estimation layer itself has the desired $\sigma_1^q$ scaling. What remains is a high-probability accelerated outer method that is robust to the resulting constant relative error without adding unfavorable conditioning requirements. Existing relative-inexact acceleration analyses impose additional restrictions \citep{vasin2023relative,kornilov2025intermediate}. We therefore do not state a general finite-$p$ matching theorem for $\sigma_1$.

\section{Gradient-Only End-to-End Initialization}\label{app:init}
The restart statements use a supplied upper bound $F(x_0)\le\Delta_0$. This can be obtained from stochastic gradients without evaluating $f$, $f^\star$, or an exact gradient.

\subsection{Choosing the auxiliary set}\label{app:choosing-Q}
The convex result requires a closed bounded convex set $Q$ with $x_0,x^\star\in Q$ and a tractable Euclidean projection; this set controls the iterates but does not constrain the original problem. One may take a known compact convex parameter set containing $x_0$ and an optimizer, or $Q=B(x_0,R)$ when a radius $R\ge\|x_0-x^\star\|$ is available, in which case $D=2R$. Regularization can provide such a radius: if $x_0=0$ and $f(x)=\ell(x)+(\lambda/2)\|x\|^2$ for a nonnegative convex loss $\ell$, then $f(x^\star)\le f(0)$ gives $\|x^\star\|\le\sqrt{2\ell(0)/\lambda}$. Here $f$ is also $\lambda$-strongly convex, so the shrinking-ball construction of Appendix~\ref{app:restart} applies without prior knowledge of $x^\star$.

\subsection{Robust gradient estimation}

\begin{lemma}[Finite-$p$ geometric median-of-means]\label{lem:GMoM}
Let $X_i$ be independent, $\E X_i=\mu$, and $\E\|X_i-\mu\|^p\le\nu^p$ for $1<p\le2$. For $n\ge c\log(2/\delta)$, a geometric median-of-means estimator $\widehat\mu$ satisfies with probability at least $1-\delta$,
\[
 \|\widehat\mu-\mu\|\le K_p\nu\left(\frac{\log(2/\delta)}{n}\right)^{(p-1)/p}.
\]
\end{lemma}
\begin{proof}
We divide the $n$ independent observations into independent blocks, control each block mean by a finite-$p$ moment inequality and Markov's inequality, amplify the fraction of good blocks by a Chernoff bound, and finally use geometric-median stability to obtain the stated confidence bound.

\paragraph{Block means.}
Let $k=\lceil c_0\log(2/\delta)\rceil$ and split the observations into $k$ disjoint blocks of common size
$b=\lfloor n/k\rfloor$, discarding at most $k-1$ observations. For block $j$, let
\[
 \bar X_j:=\frac1b\sum_{i\in\mathcal I_j}X_i.
\]
The finite-$p$ moment inequality for sums of independent mean-zero random vectors in a Hilbert space gives
\[
 \E\|\bar X_j-\mu\|^p
 \le C_p\nu^p b^{1-p}.
\]
This is the vector analogue of the von Bahr--Esseen inequality; see, for example, \citet{minsker2015geom}. Markov's inequality therefore implies
\[
 \Prb\left\{
 \|\bar X_j-\mu\|>
 (4C_p)^{1/p}\nu b^{-(p-1)/p}
 \right\}\le\frac14.
\]
\paragraph{Probability amplification.}
The blocks are independent. A Chernoff bound shows that, for a sufficiently large absolute choice of $c_0$, with probability at least $1-\delta$ more than half of the block means lie in the ball centered at $\mu$ with radius
\[
 r_p:=(4C_p)^{1/p}\nu b^{-(p-1)/p}.
\]
\paragraph{Geometric-median aggregation.}
On this event, the geometric-median stability lemma of
\citet{minsker2015geom} implies that any geometric median of the block means is at distance at most a universal constant times $r_p$ from $\mu$. Since $b\ge c_1n/k$ when $n\ge c k$, we conclude that
\[
 \|\widehat\mu-\mu\|
 \le K_p\nu
 \left(\frac{k}{n}\right)^{(p-1)/p}
 \le K_p\nu
 \left(\frac{\log(2/\delta)}{n}\right)^{(p-1)/p}.
\]
All constants depending on the Hilbert-space moment inequality and on the geometric-median stability factor are absorbed into $K_p$.
\end{proof}

Let
$\kappa=K_p(\log(2/\delta)/n_0)^{(p-1)/p}$ and $\widehat r_0=\|\widehat g_0\|$.

\subsection{Convex initialization}
Assume the auxiliary set $Q$ is already known and contains both $x_0$ and $x^\star$. Define
$r_0:=\|\nabla f(x_0)\|$.
\paragraph{Gap bound from convexity.}
Convexity at $x_0$ gives
\[
 F(x_0)
 =f(x_0)-f(x^\star)
 \le\ip{\nabla f(x_0)}{x_0-x^\star}
 \le r_0\|x_0-x^\star\|
 \le Dr_0.
\]
\paragraph{Conditional moment bound.}
Consequently, Assumption~\ref{ass:oracle} at $x_0$ gives the $p$th-moment scale
\begin{align*}
 \left(
 \sigma_0^p+\sigma_1^pr_0^p
 +\sigma_2^pF(x_0)^{p/2}
 \right)^{1/p}
 &\le
 \left(
 \sigma_0^p+\sigma_1^pr_0^p
 +\sigma_2^p(Dr_0)^{p/2}
 \right)^{1/p}\le\sigma_0+\sigma_1r_0+\sigma_2\sqrt{Dr_0},
\end{align*}
where the last step is the triangle inequality for the $\ell_p$ norm of three nonnegative scalars. On the good event from Lemma~\ref{lem:GMoM},
\[
 \|\widehat g_0-\nabla f(x_0)\|
 \le\kappa\left(
 \sigma_0+\sigma_1r_0+\sigma_2\sqrt{Dr_0}
 \right).
\]
\paragraph{Observable upper bound.}
The triangle inequality $r_0\le\|\widehat g_0\|+\|\widehat g_0-\nabla f(x_0)\|$ therefore yields
\[
 r_0\le \widehat r_0+\kappa\sigma_0+\kappa\sigma_1r_0+\kappa\sigma_2\sqrt{Dr_0}.
\]
\paragraph{Quadratic inequality.}
If $\kappa\sigma_1<1$, define
$\alpha=1-\kappa\sigma_1$, $\beta=\kappa\sigma_2\sqrt D$, and $\gamma=\widehat r_0+\kappa\sigma_0$. With $y=\sqrt{r_0}$,
$\alpha y^2-\beta y-\gamma\le0$. Since $\alpha>0$ and $y\ge0$, the quadratic inequality places $y$ below its nonnegative root:
\[
 y\le\frac{\beta+\sqrt{\beta^2+4\alpha\gamma}}{2\alpha}.
\]
Squaring and using $F(x_0)\le Dr_0$ gives
\begin{equation}\label{eq:init-cvx}
 r_0\le \bar r_0^{\rm cvx}:=
 \left(\frac{\beta+\sqrt{\beta^2+4\alpha\gamma}}{2\alpha}\right)^2,
 \qquad \widehat\Delta_0^{\rm cvx}:=D\bar r_0^{\rm cvx}.
\end{equation}
Thus \eqref{eq:init-cvx} makes $\widehat\Delta_0^{\rm cvx}$ a computable upper bound on $F(x_0)$ on the geometric-median-of-means event. It suffices to choose $n_0$ such that
$\kappa\sigma_1\le1/2$, an accuracy-independent condition that keeps the denominator uniformly away from zero.

\subsection{Strongly convex initialization}
\paragraph{Gap bound from strong convexity.}
For a differentiable $\mu$-strongly convex function, Assumption~\ref{ass:convexity} with $x=x_0$ and $y=x^\star$ gives
\[
 F(x_0)
 \le\ip{\nabla f(x_0)}{x_0-x^\star}
 -\frac{\mu}{2}\|x_0-x^\star\|^2.
\]
Maximizing the right-hand side over
$s=\|x_0-x^\star\|\ge0$ and using Cauchy--Schwarz yields
\[
 F(x_0)\le\max_{s\ge0}
 \left(r_0s-\frac{\mu}{2}s^2\right)
 =\frac{r_0^2}{2\mu}.
\]
\paragraph{Effective relative noise.}
Define the effective relative coefficient
\begin{equation}\label{eq:sig1eff}
 \bar\sigma_{1,\mu}:=
 \left(\sigma_1^p+\frac{\sigma_2^p}{(2\mu)^{p/2}}\right)^{1/p}.
\end{equation}
The definition in \eqref{eq:sig1eff} combines the two oracle-moment terms that are proportional to $r_0^p$.
The oracle moment at $x_0$ is then bounded by
\begin{align*}
 \sigma_0^p+\sigma_1^pr_0^p+\sigma_2^pF(x_0)^{p/2}
 &\le\sigma_0^p+
 \left(\sigma_1^p+
 \frac{\sigma_2^p}{(2\mu)^{p/2}}\right)r_0^p\\
 &=\sigma_0^p+\bar\sigma_{1,\mu}^pr_0^p.
\end{align*}
Taking the $p$th root, applying Lemma~\ref{lem:GMoM}, and using the triangle inequality gives
\[
 r_0\le\widehat r_0+\kappa\sigma_0
 +\kappa\bar\sigma_{1,\mu}r_0.
\]
\paragraph{Observable upper bound.}
If $\kappa\bar\sigma_{1,\mu}<1$, rearrangement yields
\begin{equation}\label{eq:init-sc}
 \bar r_0^{\rm sc}:=\frac{\widehat r_0+\kappa\sigma_0}{1-\kappa\bar\sigma_{1,\mu}},
 \qquad
 \widehat\Delta_0^{\rm sc}:=\frac{(\bar r_0^{\rm sc})^2}{2\mu}.
\end{equation}
Equation~\ref{eq:init-sc} and the first inequality above then give
$F(x_0)\le\widehat\Delta_0^{\rm sc}$ on the geometric-median-of-means event. Splitting the total failure probability, for example assigning $\rho/2$ to initialization and $\rho/2$ to the restart run, gives an end-to-end success probability at least $1-\rho$. The construction uses sampled stochastic gradients and known upper bounds on the model parameters, but no objective values, value of $f^\star$, or exact gradients.

\section{Endpoints and Special Cases}\label{app:endpoints}
\subsection{$p=2$}
When $p=2$, the conjugate exponent is $q=2$. Assumption~\ref{ass:oracle} becomes
\[
 \E_t\|g(x,\xi_t)-\nabla f(x)\|^2\le\sigma_0^2+\sigma_1^2\|\nabla f(x)\|^2+\sigma_2^2F(x).
\]
No step in Lemma~\ref{lem:clipping} becomes singular: its clipping-bias bound is
$2\varsigma_{2,\Delta}^2/\Lambda$, and its conditional second-moment bound is
$2\varsigma_{2,\Delta}^2$.

\paragraph{Convex specialization.}
Substituting $q=2$ into Theorem~\ref{thm:convex} gives
\[
 \widetilde O_{\rho}\!\left[
 \begin{aligned}
 &S+D\left(\sqrt{\frac{H_0}{\varepsilon}}
 +\sqrt{H_1}\,J_H\right)
 +\left(\frac{\sigma_0D}{\varepsilon}\right)^2\\
 &\quad
 +\left(\sigma_1D\left[
 \sqrt{\frac{H_0}{\varepsilon}}
 +\sqrt{H_1}\,J_H^{1/2}
 \right]\right)^2
 +\left(\frac{\sigma_2D}{\sqrt\varepsilon}\right)^2
 \end{aligned}
 \right].
\]
Thus the three pure stochastic accuracy scales are
$\varepsilon^{-2}$ for residual noise and
$\varepsilon^{-1}$ for the gradient- and gap-dependent components, while the displayed square retains the $H_0$--$H_1$ cross term.

\paragraph{Strongly convex specialization.}
Theorem~\ref{thm:strong} specializes to
\[
 \widetilde O_{\rho}\!\left[
 \left(1+\sqrt{\frac{H_0}{\mu}}\right)S
 +\sqrt{\frac{H_1\Delta_0}{\mu}}
 +\frac{\sigma_0^2}{\mu\varepsilon}
 +\frac{\sigma_1^2}{\mu}
 \left(\sqrt{H_0S}+\sqrt{H_1\Delta_0}\right)^2
 +\frac{\sigma_2^2}{\mu}S
 \right].
\]
Only the $\sigma_0$ term is polynomial in the target accuracy; the terms multiplied by $S$ are logarithmic because
$S=\lceil\log_2(\Delta_0/\varepsilon)\rceil$.

\subsection{Norm-sub-Gaussian endpoint}
If instead the centered oracle is conditionally norm-sub-Gaussian with a state-dependent proxy comparable to
$\sigma_0^2+\sigma_1^2\|\nabla f(x)\|^2+\sigma_2^2F(x)$, localization again creates the phase proxy
$\nu_{\rm SG,\Delta}^2=\sigma_0^2+\sigma_1^2G_\Delta^2+4\sigma_2^2\Delta$.
The finite-$p$ clipping proof remains applicable with $p=2$ and already gives the accuracy exponents displayed above. An unclipped mini-batch alternative would require a separately stated vector sub-Gaussian mean-concentration inequality, including its precise definition of the proxy and its numerical constants. Because those data are not part of the present assumptions, no theorem in this paper relies on that alternative; we record it only as a possible implementation endpoint.

\subsection{Deterministic limit and other audit cases}
\paragraph{No stochastic noise.}
If $\sigma_0=\sigma_1=\sigma_2=0$, then
$g(x,\xi_t)=\nabla f(x)$ almost surely under
Assumption~\ref{ass:oracle}. Hence
$\zeta_t=\beta_t=0$, the two martingale channels vanish, and
\eqref{eq:pathwise} reduces to $\Phi_t\le\Phi_{t-1}$. The restart sums retain only the constant and deterministic curvature terms.

\paragraph{No residual noise.}
If $\sigma_0=0$, the additive term in
\eqref{eq:app-phase-rate} disappears. In the strongly convex case, the remaining $\sigma_1$ contribution is the sum of a constant per phase and a decreasing geometric sequence, while the $\sigma_2$ contribution is constant per phase. Thus the target-accuracy dependence is logarithmic through $S$, with no term proportional to a positive power of $1/\varepsilon$.

\paragraph{No gradient-dependent noise.}
If $\sigma_1=0$, the gradient-dependent term is absent from
$\varsigma_{p,\Delta}$. The crossover calculation with exponent $q/2$ is then unnecessary; the stochastic complexity contains only the additive and gap-dependent geometric sums.

\paragraph{No gap-dependent noise.}
If $\sigma_2=0$, the term
$(\sigma_2D/\sqrt\Delta)^q$ vanishes in every phase. The analysis and clipping threshold remain valid without modification because all inequalities were upper bounds by sums of nonnegative noise components.

\paragraph{Globally smooth endpoint.}
If $H_1=0$, then $r_H=\infty$, $C_\star=2$, and
$\mathcal H_\Delta=H_0$. Every local-model radius check becomes automatic, $J_H=0$, and the deterministic part reduces to the classical accelerated dependence
$D\sqrt{H_0/\varepsilon}$ in the convex case and
$\sqrt{H_0/\mu}\,S$ in the strongly convex case. Since our convention is $L=2H_0$, this is the usual globally $L$-smooth subclass up to the fixed factor $\sqrt2$.

These reductions verify that the unified formulas remain well-defined at every endpoint and that no proof step divides by a parameter that may be zero.

\section{Horizon-Driven Parameter-Free Inner Method}\label{app:parameterfree}
This section removes $p$, $\sigma_0,\sigma_1,\sigma_2,H_0,H_1$ from the \emph{runtime inputs} of a phase. They remain in an a posteriori sufficient horizon condition. The phase still receives $(w,\Delta,Q,D,T,\rho)$, and the restart layer still needs a valid phase level; strong convexity uses $\mu$ to form shrinking balls.

\subsection{A radius-aware self-bound}
For analysis only, define
\begin{equation}\label{eq:Gamma}
 \Gamma_\Delta:=H_0+4H_1\Delta+\frac{4\Delta}{r_H^2},
\end{equation}
with $r_H^{-2}=0$ if $r_H=\infty$.
The quantity in \eqref{eq:Gamma} upper-bounds every curvature and radius term that appears after localization to $F(x)\le4\Delta$.
\begin{lemma}[Radius-aware self-bound]\label{lem:radius-aware}
For every $x$,
\[
 \|\nabla f(x)\|^2\le4F(x)\left(H_0+H_1F(x)+\frac{F(x)}{r_H^2}\right).
\]
\end{lemma}
\begin{proof}
Fix $x$ and abbreviate
$g:=\|\nabla f(x)\|$ and $H_x:=\mathcal H(x)$. We compare the unconstrained quadratic-model step with the local-model radius, use the shorter of the two admissible steps, and then handle the zero-curvature cases separately. If $g=0$, the claim is immediate. Suppose first that $H_x>0$ and the unconstrained quadratic-model step lies inside the validity radius:
\[
 \frac{g}{2H_x}\le r_H.
\]
\paragraph{Interior quadratic-model step.}
Set
$y=x-\nabla f(x)/(2H_x)$. Then
$\|y-x\|=g/(2H_x)\le r_H$, so Assumption~\ref{ass:h01} applies and gives
\begin{align*}
 f(y)
 &\le f(x)
 -\frac{g^2}{2H_x}
 +H_x\frac{g^2}{4H_x^2}
 =f(x)-\frac{g^2}{4H_x}.
\end{align*}
Since $f^\star\le f(y)$, it follows that
$F(x)\ge g^2/(4H_x)$, or
\[
 g^2\le4F(x)H_x.
\]

\paragraph{Full-radius step.}
It remains to treat the case in which the quadratic-model step would be longer than $r_H$, namely
$g/(2H_x)>r_H$. Take the full-radius step
\[
 y=x-r_H\frac{\nabla f(x)}{g}.
\]
The local model and $f^\star\le f(y)$ give
\[
 F(x)\ge r_Hg-H_xr_H^2.
\]
The case condition implies $H_xr_H<g/2$, and hence
\[
 F(x)>\frac{r_Hg}{2},
 \qquad
 g^2<\frac{4F(x)^2}{r_H^2}.
\]
The first case is bounded by the first term and the second case by the second term in
\[
 4F(x)\left(H_x+\frac{F(x)}{r_H^2}\right).
\]
This proves the result when $H_x>0$.

\paragraph{Degenerate curvature cases.}
If $H_x=0$ and $r_H<\infty$, the same full-radius step gives
$F(x)\ge r_Hg$ and therefore the stronger bound
$g^2\le F(x)^2/r_H^2$. If $H_x=0$ and $r_H=\infty$, the upper model is
$f(y)\le f(x)+\ip{\nabla f(x)}{y-x}$ for every $y$. Choosing
$y=x-t\nabla f(x)$ and letting $t\to\infty$ would force
$f^\star=-\infty$ unless $g=0$. Since a minimizer exists, $g=0$. Thus all degenerate cases are covered.
\end{proof}

\subsection{Runtime algorithm}
Set as before $A_0=D^2/\Delta$, $A_{\rm tar}=6D^2/\Delta$, and for a user-chosen horizon $T$,
\[
 a=\frac{5D^2}{\Delta T},\qquad A_t=A_0+ta,
 \qquad \ell_\rho=\log\frac8\rho,
 \qquad \Lambda_{\Delta,T}=\frac{\Delta T}{128D\ell_\rho}.
\]
The runtime update is exactly Procedure~\ref{proc:phase}, but $\Lambda$ is replaced by $\Lambda_{\Delta,T}$ and no hidden model parameters are read.

\begin{lemma}[Horizon-induced localization]\label{lem:pf-local}
If $\Phi_{t-1}\le3D^2$ and
\begin{equation}\label{eq:pf-geom-horizon}
 T\ge40D\sqrt{\frac{\Gamma_\Delta}{\Delta}},
\end{equation}
then $F(y_{t-1})\le3\Delta$, $F(u_t)<4\Delta$, both local-model displacements are at most $r_H$,
$\mathcal H(u_t)a^2/A_t\le1/64$, and
\[
 \|\nabla f(u_t)\|\le4\sqrt{\Delta\Gamma_\Delta}.
\]
\end{lemma}
\begin{proof}
Fix $t\in\{1,\ldots,T\}$ and assume
$\Phi_{t-1}\le3D^2$ together with the horizon condition
\eqref{eq:pf-geom-horizon}. We first control the previous output gap and both accelerated displacements, then apply the local upper model to the query point, and finally verify the quadratic coefficient and the query-gradient bound.

\paragraph{Output gap and displacements.}
The potential bound and $A_{t-1}\ge A_0=D^2/\Delta$ first give
\[
 F(y_{t-1})
 \le\frac{\Phi_{t-1}}{A_{t-1}}
 \le\frac{3D^2}{D^2/\Delta}
 =3\Delta.
\]
Next, $a/A_t\le a/A_0=5/T$. Since projection keeps every $z_j$ in $Q$, and each $y_j$ is a convex combination of points in $Q$, all $y_j,z_j$ belong to $Q$. Therefore
\begin{align}
 \|u_t-y_{t-1}\|
 &=\frac{a}{A_t}\|z_{t-1}-y_{t-1}\|
 \le\frac{5D}{T},
 \label{eq:pf-first-displacement}\\
 \|y_t-u_t\|
 &=\frac{a}{A_t}\|z_t-z_{t-1}\|
 \le\frac{5D}{T}.
 \label{eq:pf-second-displacement}
\end{align}
Under \eqref{eq:pf-geom-horizon},
\[
 \frac{5D}{T}
 \le\frac18\sqrt{\frac{\Delta}{\Gamma_\Delta}}.
\]
Because $\Gamma_\Delta\ge4\Delta/r_H^2$, the last quantity is at most
$r_H/16$. Thus both displacements in
\eqref{eq:pf-first-displacement}--\eqref{eq:pf-second-displacement}
are within the radius of Assumption~\ref{ass:h01}.

\paragraph{Query gap.}
At $y_{t-1}$, Lemma~\ref{lem:radius-aware} and
$F(y_{t-1})\le3\Delta$ give
\begin{align}
 \|\nabla f(y_{t-1})\|^2
 &\le4F(y_{t-1})
 \left(
 H_0+H_1F(y_{t-1})
 +\frac{F(y_{t-1})}{r_H^2}
 \right)\le12\Delta\Gamma_\Delta.
 \label{eq:pf-gradient-at-output}
\end{align}
Apply the local upper model from $y_{t-1}$ to $u_t$. Using
\eqref{eq:pf-first-displacement},
\eqref{eq:pf-gradient-at-output}, and
$\mathcal H(y_{t-1})\le\Gamma_\Delta$, we obtain
\begin{align*}
 F(u_t)
 &\le F(y_{t-1})
 +\|\nabla f(y_{t-1})\|
 \|u_t-y_{t-1}\|
 +\mathcal H(y_{t-1})
 \|u_t-y_{t-1}\|^2\\
 &\le3\Delta
 +\sqrt{12\Delta\Gamma_\Delta}
 \frac18\sqrt{\frac{\Delta}{\Gamma_\Delta}}
 +\Gamma_\Delta
 \frac1{64}\frac{\Delta}{\Gamma_\Delta}=\left(3+\frac{\sqrt{12}}8+\frac1{64}\right)\Delta
 <4\Delta.
\end{align*}
Consequently,
$\mathcal H(u_t)=H_0+H_1F(u_t)
\le H_0+4H_1\Delta\le\Gamma_\Delta$.

\paragraph{Quadratic coefficient and query gradient.}
Finally,
\[
 \frac{a^2}{A_t}
 \le\frac{a^2}{A_0}
 =\frac{25D^2}{\Delta T^2}.
\]
Combining this inequality with
$T^2\ge1600D^2\Gamma_\Delta/\Delta$ gives
\[
 \frac{\mathcal H(u_t)a^2}{A_t}
 \le\frac{25D^2\Gamma_\Delta}{\Delta T^2}
 \le\frac1{64}.
\]
Applying Lemma~\ref{lem:radius-aware} once more, now at $u_t$, and using
$F(u_t)<4\Delta$, yields
\[
 \|\nabla f(u_t)\|^2
 \le4F(u_t)\Gamma_\Delta
 <16\Delta\Gamma_\Delta.
\]
Taking square roots proves the stated gradient bound and completes the localization argument.
\end{proof}

Assume the actual but unknown oracle parameters obey Assumption~\ref{ass:oracle}. Define only for analysis
\begin{equation}\label{eq:nubar}
 \bar\varsigma_{p,\Delta}^p
 :=\sigma_0^p+\bigl(4\sigma_1\sqrt{\Delta\Gamma_\Delta}\bigr)^p
 +\sigma_2^p(4\Delta)^{p/2}.
\end{equation}

\begin{theorem}[Horizon-driven parameter-free contraction]\label{thm:pf}
The runtime trajectory described above does not use $p$, $\sigma_i$, or $H_i$. Nevertheless, if the actual parameters satisfy the assumptions and
\begin{equation}\label{eq:pf-horizon}
 T\ge\max\left\{
 1024\ell_\rho D\sqrt{\frac{\Gamma_\Delta}{\Delta}},\ 
 128\,160^{1/(p-1)}\ell_\rho
 \left(\frac{\bar\varsigma_{p,\Delta}D}{\Delta}\right)^q
 \right\},
\end{equation}
then $\Prb\{F(y_T)\le\Delta/2\}\ge1-\rho$.
\end{theorem}

\begin{proof}
Fix the phase data $(w,\Delta,Q,D,T,\rho)$ and the filtration from
Assumption~\ref{ass:oracle}. Recall that
$a=5D^2/(\Delta T)$,
$A_t=D^2/\Delta+ta$, and
$\Lambda_{\Delta,T}=\Delta T/(128D\ell_\rho)$.
At iteration $t$, write
\[
 \beta_t:=\E_t\widehat g_t-\nabla f(u_t),
 \qquad
 \zeta_t:=\widehat g_t-\E_t\widehat g_t,
 \qquad
 v_t:=\E_t\|\zeta_t\|^2.
\]
The proof first stops the potential before localization can fail, then verifies the clipping bias and conditional moment bounds on the stopped path, applies Freedman's inequality to the two centered channels, and finally removes the stopping time by contradiction.

\paragraph{Stopped process and localization.}
Set
\[
 \tau:=\inf\{t\in\{1,\ldots,T\}:\Phi_t>3D^2\},
 \qquad
 I_t:=\mathbf1\{t-1<\tau\},
\]
with $\inf\varnothing=T+1$. The potential $\Phi_j$ is
$\mathcal F_j$-measurable, and
\[
 \{t-1<\tau\}
 =\bigcap_{j=1}^{t-1}\{\Phi_j\le3D^2\}
 \in\mathcal F_{t-1}.
\]
Therefore $I_t$ is predictable. Moreover, $I_t=1$ implies that the localization premise
$\Phi_{t-1}\le3D^2$ holds. The first condition in
\eqref{eq:pf-horizon} is stronger than
\eqref{eq:pf-geom-horizon}, so Lemma~\ref{lem:pf-local} applies before $\tau$.

\paragraph{Clipping radius and conditional moment bound.}
The same horizon condition also makes the clipping radius large enough. Indeed,
\[
 2\|\nabla f(u_t)\|\le8\sqrt{\Delta\Gamma_\Delta}
 \le\frac{\Delta T}{128D\ell_\rho}=\Lambda_{\Delta,T},
\]
where the second inequality is exactly
$T\ge1024\ell_\rho D\sqrt{\Gamma_\Delta/\Delta}$ after rearrangement. Hence
$\Lambda_{\Delta,T}\ge2\|\nabla f(u_t)\|$, and the proof of Lemma~\ref{lem:clipping} applies with the localized moment scale
$\bar\varsigma_{p,\Delta}$ from \eqref{eq:nubar}.

\paragraph{Clipping bias and centered second moment.}
The second horizon condition implies
\begin{equation}\label{eq:pf-bias-cond}
 \Lambda_{\Delta,T}^{p-1}\ge160\bar\varsigma_{p,\Delta}^p\frac{D}{\Delta}.
\end{equation}
To verify \eqref{eq:pf-bias-cond}, the second term in
\eqref{eq:pf-horizon} and the definition of the clipping radius give
\begin{align*}
 \Lambda_{\Delta,T}
 &=\frac{\Delta T}{128D\ell_\rho}\ge160^{1/(p-1)}\frac{\Delta}{D}
 \left(\frac{\bar\varsigma_{p,\Delta}D}{\Delta}\right)^q=\left(
 160\bar\varsigma_{p,\Delta}^p\frac{D}{\Delta}
 \right)^{1/(p-1)}.
\end{align*}
The last equality uses $q=p/(p-1)$, so raising both sides to the positive power $p-1$ proves \eqref{eq:pf-bias-cond}. Lemma~\ref{lem:clipping} and
\eqref{eq:pf-bias-cond} now give, before $\tau$,
\[
 \|\beta_t\|
 \le2\bar\varsigma_{p,\Delta}^p
 \Lambda_{\Delta,T}^{1-p}
 \le\frac{\Delta}{80D},
\]
and
\[
 v_t:=\E_t\|\zeta_t\|^2
 \le2\bar\varsigma_{p,\Delta}^p
 \Lambda_{\Delta,T}^{2-p}
 \le\frac{\Lambda_{\Delta,T}\Delta}{80D}.
\]

\paragraph{Stopped potential recursion.}
Lemma~\ref{lem:pf-local} gives
$\mathcal H(u_t)a^2/A_t\le1/64<1/16$, so the coefficient condition of Lemma~\ref{lem:one-step} holds. The two displacements are within $r_H$ by Lemma~\ref{lem:pf-local}, so both uses of the local upper model in that lemma are also valid. Multiplying its pathwise recursion by the predictable indicator $I_t$, using
$\|\zeta_t+\beta_t\|^2\le2\|\zeta_t\|^2+2\|\beta_t\|^2$, and adding and subtracting $v_t$ gives exactly the stopped recursion \eqref{eq:app-stopped-rec}.

\paragraph{Predictable contributions.}
We now verify each numerical channel for the new threshold. Since
$Ta=5D^2/\Delta$,
\[
 \sum_{t=1}^T aI_tD\|\beta_t\|
 \le TaD\frac{\Delta}{80D}
 =\frac{D^2}{16},
\]
while
\[
 2\sum_{t=1}^T a^2I_t\|\beta_t\|^2
 \le2Ta^2\frac{\Delta^2}{80^2D^2}
 =\frac{D^2}{128T}.
\]
For the conditional-variance contribution, use
$\Lambda_{\Delta,T}=\Delta T/(128D\ell_\rho)$:
\begin{align*}
 2\sum_{t=1}^T a^2I_tv_t
 &\le2Ta^2
 \frac{\Lambda_{\Delta,T}\Delta}{80D}=\frac{5D^2}{1024\ell_\rho}
 <\frac{D^2}{400},
\end{align*}
where the last inequality follows from
$\ell_\rho=\log(8/\rho)>\log8$.

\paragraph{Martingale terms.}
For $j\le T$, define
\[
 M_j^{(1)}:=-\sum_{t=1}^j aI_t
 \ip{\zeta_t}{z_{t-1}-x^\star},
 \qquad
 M_j^{(2)}:=2a^2\sum_{t=1}^jI_t
 (\|\zeta_t\|^2-v_t).
\]
The point $u_t$, the iterate $z_{t-1}$, and the indicator $I_t$ are
$\mathcal F_{t-1}$-measurable. By construction,
$\E_t\zeta_t=0$ and
$\E_t(\|\zeta_t\|^2-v_t)=0$. Thus both displayed processes are scalar martingales adapted to $(\mathcal F_j)_{j\ge0}$.
Moreover, clipping gives $\|\widehat g_t\|\le\Lambda_{\Delta,T}$ almost surely and Jensen's inequality gives
$\|\E_t\widehat g_t\|\le\Lambda_{\Delta,T}$; hence
$\|\zeta_t\|\le2\Lambda_{\Delta,T}$ almost surely.

For the signed martingale, $\|z_{t-1}-x^\star\|\le D$ and the localized bound on $v_t$ imply
\begin{align*}
 |\Delta M_t^{(1)}|
 &\le2a\Lambda_{\Delta,T}D
 =\frac{5D^2}{64\ell_\rho},\\
 V_T^{(1)}
 &:=\sum_{t=1}^T\E_t[(\Delta M_t^{(1)})^2]\le a^2D^2\sum_{t=1}^TI_tv_t
 \le a^2D^2T\frac{\Lambda_{\Delta,T}\Delta}{80D}
 =\frac{5D^4}{2048\ell_\rho}.
\end{align*}
For the quadratic martingale,
$|\|\zeta_t\|^2-v_t|\le4\Lambda_{\Delta,T}^2$ and
$\|\zeta_t\|^4\le4\Lambda_{\Delta,T}^2\|\zeta_t\|^2$ almost surely. Consequently,
\begin{align*}
 |\Delta M_t^{(2)}|
 &\le8a^2\Lambda_{\Delta,T}^2
 =\frac{25D^2}{2048\ell_\rho^2},\\
 V_T^{(2)}
 &:=\sum_{t=1}^T\E_t[(\Delta M_t^{(2)})^2]\le16a^4\Lambda_{\Delta,T}^2\sum_{t=1}^TI_tv_t
 \le16a^4\Lambda_{\Delta,T}^2T
 \frac{\Lambda_{\Delta,T}\Delta}{80D}=\frac{125D^4}{2097152\ell_\rho^3}.
\end{align*}

\paragraph{Freedman bound and one-phase contraction.}
We apply the two-sided maximal Freedman inequality stated in Appendix~\ref{app:phaseproof} with $u=D^2/8$. For the signed martingale, the denominator in the Freedman exponent satisfies
\begin{align*}
 2\left(V_T^{(1)}+\frac{L_1u}{3}\right)
 &\le\frac{2D^4}{\ell_\rho}
 \left(\frac5{2048}+\frac5{1536}\right)=\frac{35D^4}{3072\ell_\rho}
 \le\frac{D^4}{64\ell_\rho}
 =\frac{u^2}{\ell_\rho}.
\end{align*}
For the quadratic martingale, $\ell_\rho>1$ gives
\begin{align*}
 2\left(V_T^{(2)}+\frac{L_2u}{3}\right)
 &\le\frac{2D^4}{\ell_\rho^2}
 \left(\frac{125}{2097152}+\frac{25}{49152}\right)\le\frac{D^4}{64\ell_\rho}
 =\frac{u^2}{\ell_\rho}.
\end{align*}
The second numerical inequality holds because
$2(125/2097152+25/49152)<1/64$.
Thus the exponent $u^2/[2(V+Lu/3)]$ is at least
$\ell_\rho$ for each martingale, and each failure probability is at most
$2e^{-\ell_\rho}=\rho/4$. On their common good event, which has probability at least $1-\rho/2$, the stopped recursion and
\eqref{eq:app-initial-potential} give, uniformly in $j\le T$,
\[
 \Phi_{j\wedge\tau}\le\frac32D^2+\frac{D^2}{16}+\frac{D^2}{128T}+\frac{D^2}{400}+\frac{D^2}{8}+\frac{D^2}{8}<3D^2,
\]
where the terms after $3D^2/2$ are, respectively, the accumulated bias inner product, squared bias, predictable variance, and the two martingale deviations. If $\tau\le T$, choosing $j=\tau$ contradicts the definition of $\tau$. Hence $\tau=T+1$ on the good event. Finally,
$A_T=A_0+Ta=6D^2/\Delta$ and
$A_TF(y_T)\le\Phi_T\le3D^2$, so
$F(y_T)\le\Delta/2$. Since the good event has probability at least $1-\rho/2\ge1-\rho$, the theorem follows.
\end{proof}

\begin{remark}[Meaning of ``parameter-free'']
Theorem~\ref{thm:pf} is fixed-budget adaptation. The unknown parameters determine which horizons satisfy \eqref{eq:pf-horizon}; the algorithm does not certify from the observed trajectory that a particular unknown sufficient horizon has already been reached. A common-budget restart wrapper can allocate $T=\lfloor N/S\rfloor$ to every phase, losing at most another factor $S$ relative to phase-wise tuning. In the strongly convex wrapper, $\mu$ remains an operational parameter because it determines the shrinking sets $Q_s$.
\end{remark}

\section{Experimental Details and Additional Results}\label{app:experiments}
\paragraph{Construction.}
We sample 64 vehicle and 64 animal images from the CIFAR-10 training set and extract frozen DINOv2 ViT-S/14 features. The encoder is never updated; optimization acts only on a convex kernel head, the dual counterpart of last-layer fitting on fixed representations. After row normalization, let $G$ be the feature Gram matrix. We choose a ridge $\gamma\ge0$ and a scalar $c>0$ so that $K=(G+\gamma I)/c$ has largest eigenvalue one and condition number 20. The labels are $y_i=1$ for vehicles and $y_i=-1$ for animals. For fixed weights $\omega_i>0$, the component objective and gradient are
\[
 f_i(\alpha)=\omega_i[\cosh((K\alpha)_i-y_i)-1],\qquad g_i(\alpha)=\omega_i\sinh((K\alpha)_i-y_i)K_{i:}.
\]
Because $K$ is invertible, $\alpha^\star=K^{-1}y$ gives $f_i(\alpha^\star)=0$ and $g_i(\alpha^\star)=0$ for every $i$. Furthermore,
\[
 \nabla^2 f(\alpha)=\frac1nK^\top\operatorname{diag}\!\bigl(\omega_i\cosh((K\alpha)_i-y_i)\bigr)K,
\]
so $f$ is $\mu_{\rm exp}$-strongly convex with $\mu_{\rm exp}:=\min_i\omega_i\lambda_{\min}(K)^2/n$, and $\|\nabla^2 f(\alpha)\|\le R^2(1+f(\alpha))$ for $R:=\max_i\|K_{i:}\|$ because $n^{-1}\sum_i\omega_i=1$. Thus the experiment has exact interpolation and an $(H_0,H_1)$ Hessian bound with $H_0=H_1=R^2$.

In the equal-weight regime, $\omega_i=1$. In the power-law regime, the weights are drawn once from a Pareto distribution with shape $1.5$, normalized to mean one, and then stored as fixed data. Their minimum, median, 99th percentile, and maximum are $0.392$, $0.590$, $7.88$, and $9.75$, respectively. At initialization, the descriptive Hill estimate for the tail index of component-gradient norms is $1.52$. This is an empirical tail diagnostic only: a finite dataset does not generate a literally infinite-variance oracle.

\paragraph{Protocol and tuning.}
All methods start from zero and use the same streams of uniformly sampled component indices. Since every update uses one component gradient, iteration count and oracle calls coincide. Each method is tuned on seeds 101--106 and evaluated on the disjoint seeds 20001--20030. We choose among 16 configurations by final 90th-percentile relative gap, then by the median, and fix the selected configuration before evaluation. The budget is $8{,}192$, the desired absolute gap is $10^{-5}$, and RCSAG uses 16 equal-length restart phases. Its step multiplier ranges over $\{0.03,0.1,0.3,1\}$ and its fixed-budget clipping multiplier over $\{64,512,4096,32768\}$. Clipped SGD follows Algorithm~2 of \citet{gorbunov2020heavy}; its step parameter ranges over $\{0.1,0.3,1,1.8\}/L_{\rm init}$ and its clipping radius over $\{0.25,1,4,16\}$ times the median initial component-gradient norm. For both SSTM variants, the weight parameter ranges over $\{0.5,2,8,32\}N$ and the clipping parameter over $\{0.025,0.1,0.4,1.6\}$. These are practical tuning grids; they do not attempt to instantiate the conservative numerical constants of Theorem~\ref{thm:phase}.

\begin{table}[H]
\centering
\caption{Held-out results after $8{,}192$ component-gradient calls. $N_{90}$ is the first checkpoint after which at least 90\% of runs remain below $10^{-5}$; ``--'' means that no such checkpoint exists.}
\label{tab:experiment-full}
\small
\setlength{\tabcolsep}{3.5pt}
\begin{tabular}{lccccc}
\toprule
& \multicolumn{2}{c}{Equal weights} & \multicolumn{3}{c}{Power-law weights}\\
Method & Median gap & $N_{90}$ & Median gap & $N_{90}$ & Clipped\\
\midrule
RCSAG & $3.47\times10^{-7}$ & $6{,}705$ & $1.22\times10^{-6}$ & $8{,}192$ & $0.15\%$\\
Clipped SGD & $1.58\times10^{-4}$ & -- & $2.95\times10^{-4}$ & -- & $7.42\%$\\
Clipped-SSTM & $4.14\times10^{-4}$ & -- & $6.10\times10^{-4}$ & -- & $2.23\%$\\
R-clipped-SSTM & $3.21\times10^{-2}$ & -- & $3.32\times10^{-2}$ & -- & $60.51\%$\\
\bottomrule
\end{tabular}
\end{table}

\begin{figure}[H]
\centering
\includegraphics[width=0.72\textwidth]{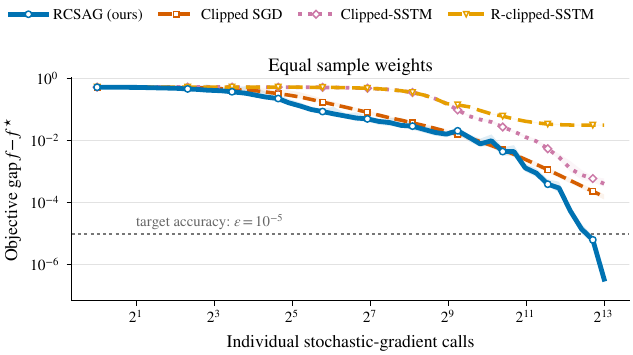}
\vspace{-0.8em}
\caption{Equal-weight control: median gap with 10th--90th percentile bands over 30 held-out seeds.}
\label{fig:real-cosh-equal}
\vspace{-0.5em}
\end{figure}

\paragraph{Additional observations.}
The equal-weight control in Figure~\ref{fig:real-cosh-equal} preserves the same ranking: RCSAG reaches $10^{-5}$ with sustained 90\% success after $6{,}705$ calls. Under fixed Pareto weights, its median gap is $1.22\times10^{-6}$ versus $2.95\times10^{-4}$ for Clipped SGD. The median clipped fraction of RCSAG changes from zero to only $0.15\%$, so the method clips rare large components rather than most updates.

\paragraph{Synthetic finite-$p$ check.}
We also rerun the comparison on the eight-dimensional strongly convex objective $f(x)=\sum_{j=1}^8\lambda_jb^{-2}[\cosh(bx_j)-1]$, where $b=1.5$ and $\lambda_j$ are geometrically spaced in $[1,16]$. The oracle adds radial Student-$t$ noise with $1.8$ degrees of freedom and scale $(\sigma_0^p+\sigma_2^pF(x)^{p/2})^{1/p}$, using $p=1.5$, $\sigma_0=0.1$, and $\sigma_2=0.4904$. Each method is selected from 16 configurations on six tuning seeds and evaluated on 300 disjoint seeds with $16{,}384$ individual gradient calls.

\begin{figure}[H]
\centering
\includegraphics[width=0.90\textwidth]{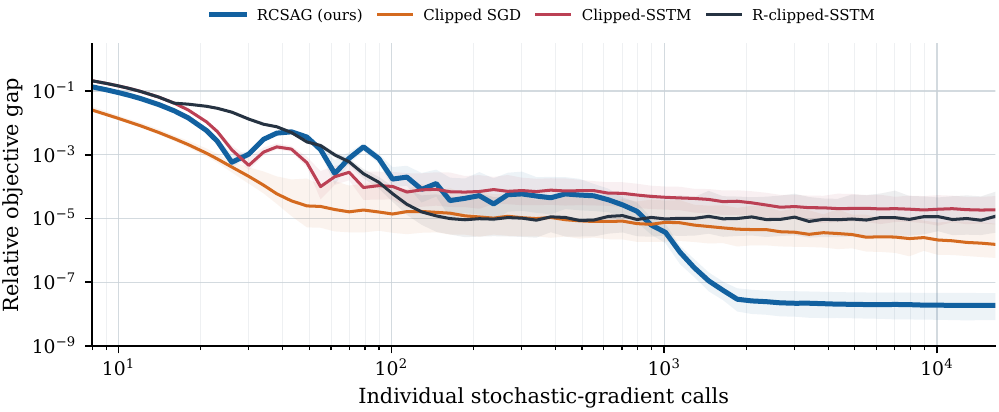}
\vspace{-0.8em}
\caption{Synthetic Student-$t$ experiment: median relative gap with 10th--90th percentile bands over 300 held-out seeds.}
\label{fig:synthetic-additive}
\vspace{-0.5em}
\end{figure}

RCSAG reaches relative gap $10^{-6}$ with sustained 90\% success after $1{,}290$ calls and succeeds on all 300 runs by the final budget. Its final 90th-percentile gap is $4.60\times10^{-8}$, compared with $6.86\times10^{-6}$ for Clipped SGD; neither SSTM variant reaches the target. The median fraction of clipped RCSAG updates is $10.1\%$. As above, these runs use tuned practical constants and do not certify the theorem's numerical constants.

\end{document}